\documentclass[smallextended]{svjour3}

\usepackage{color}
\usepackage{float}
\usepackage{bm}
\usepackage{amssymb,amsmath}
\usepackage[noadjust]{cite}
\usepackage{graphicx}
\usepackage{comment}
\usepackage{marginnote}
\usepackage{paralist}
\usepackage[normalem]{ulem}

\smartqed

\usepackage{mathrsfs}
\usepackage{mathtools}
\usepackage{algorithmic}
\usepackage{enumitem}
\allowdisplaybreaks

\makeatletter
\let\oldfootnote\footnote
\def\footnote{\@ifstar\footnote@star\footnote@nostar}
\def\footnote@star#1{{\let\thefootnote\relax\footnotetext{#1}}}
\def\footnote@nostar{\oldfootnote}
\makeatother

\usepackage[noadjust]{cite}
\usepackage{subfigure}

\floatstyle{ruled}
\newfloat{algorithm}{tbp}{loa}
\providecommand{\algorithmname}{Algorithm}
\floatname{algorithm}{\protect\algorithmname}

\newcounter{algo}

\renewcommand{\d}{\mathbf{d}}

\renewcommand{\Re}{\mathbb{R}}

\newcommand{\trt}{^{\scriptscriptstyle T}}
\newcommand{\bx}{{\mathbf{x}}}
\newcommand{\by}{{\mathbf{y}}}

\newcommand{\argmin}{\mathop{\rm argmin}}

\newcommand{\Z}{\mathbf{Z}}
\newcommand{\E}{\mathbb{E}}

\usepackage[footnotesize]{caption}
\usepackage{flushend}

\allowdisplaybreaks[4]
\makeatother

\begin{document}


\title{Asynchronous Parallel  Algorithms for Nonconvex Optimization}

\author{Loris Cannelli \and Francisco Facchinei  \and  Vyacheslav Kungurtsev \and Gesualdo Scutari}

\institute{ Loris Cannelli and Gesualdo Scutari,
School of Industrial Engineering, Purdue University, USA; \email{$<$lcannell, gscutari$>$@purdue.edu} \and
Francisco Facchinei, Department of Computer, Control, and Management Engineering Antonio Ruberti, University of Rome La Sapienza,  Roma, Italy;  \email{francisco.facchinei@uniroma1.it}\and Vyacheslav Kungurtsev
Dept. of Computer Science, Faculty of Electrical Engineering, Czech Technical University in Prague, Czech
 \email{vyacheslav.kungurtsev@fel.cvut.cz}.  The work of Cannelli and Scutari was supported by the USA NSF under Grants CIF 1564044,   CAREER Award No. 1555850, and CIF 1719205; and the Office of Naval Research (ONR) Grant N00014-16-1-2244. Facchinei was  partially supported by the Italian Ministry of Education, Research and University, under
the PLATINO (PLATform for INnOvative services in future internet) PON project,  Grant
Agreement no. PON01$\_$01007.
 Kungurtsev was supported by the Czech Science Foundation project 17-26999S.
 \newline \indent Part of this work has been presented to the  \emph{50th Asilomar Conference on Signals, Systems, and Computers}   \cite{cannelli2016parallel} and the \emph{42nd IEEE International Conference on Acoustics, Speech, and Signal Processing} \cite{cannelli2017asynchronous}. A two-part  preliminary technical report was posted on arxiv  on July 2016 \cite{companion2} and January 2017 \cite{companion}.
}

\date{\today}

\maketitle
\begin{abstract}
 We propose a   new
 asynchronous parallel block-descent algorithmic framework  for the minimization of the sum of a
 smooth nonconvex function and a   {nonsmooth convex}  one, subject to both convex and nonconvex constraints.
  The proposed  framework hinges on successive convex approximation techniques and a novel probabilistic model that  captures key elements of  modern computational architectures  and asynchronous implementations in a more faithful way than  current state-of-the-art models.
 Other key features of the framework
 are: i)  it covers in a unified way several  specific solution methods;  ii) it accommodates a variety of possible parallel computing architectures; and iii) it can deal with nonconvex constraints.
Almost sure convergence to stationary solutions  is proved, and theoretical complexity results are provided, showing  nearly ideal linear speedup when the number of workers is not too large.


  \end{abstract}
\section{Introduction}\label{sec:intro}\vspace{-0.2cm}

We  study asynchronous parallel  block-descent methods for the following class of nonconvex nonsmooth minimization problems: \vspace{-0.2cm}
\begin{equation}
\begin{array}{rl}
\underset{\mathbf{x}\triangleq (\mathbf{x}_1,\ldots, \mathbf{x}_N)}{\min} & F(\mathbf{x}) \triangleq f(\mathbf{x}) + \sum\limits_{i=1}^N g_i(\bx_i)\\[0.3em]
&\mathbf{x}_i \in \mathcal X_i, \qquad i = 1, \dots , N,
\end{array}\tag{P}\vspace{-0.1cm}
\label{eq:optimization_problem}
\end{equation}
where  $f$ is a  {smooth, possibly nonconvex function,   $g_i$ are  possibly nonsmooth, convex functions, and $\mathcal{X}_i\subseteq \Re^{n_i}$ is a  closed, possibly nonconvex  set. }

 Instances of Problem \eqref{eq:optimization_problem}  arise in
many fields, including compressed sensing,      machine learning, data mining,  and genomics,
just to name a few.
Typically, in data-driven applications $f$ might measure the misfit between the observations and the postulated model, parametrized on $\bx$, while  the regularizers $g_i$ encode structural constraints on the solution, such as sparsity.

Many of the aforementioned applications give rise to extremely
large-scale problems, which naturally call  for  \emph{asynchronous, parallel} solution methods. In fact,    well suited to
modern
computational architectures, asynchronous methods  reduce the idle times of workers,
mitigate  communication and/or memory-access congestion, and make
 algorithms
more fault-tolerant.
In this paper,  we introduce a general  asynchronous block-descent  algorithm  for finding stationary solutions of Problem \eqref{eq:optimization_problem}.

We consider a generic multi-worker architecture (e.g., shared memory system, message passing-based system,  cluster computer, cloud federation) wherein multiple workers,  
continuously
and without coordination with each other, update a block-variable   by solving a strongly convex block-model of  Problem \eqref{eq:optimization_problem}. More specifically, at iteration $k$, a worker  updates a block-variable $\bx^k_{i^k}$ of $\mathbf{x}^k$ to $\bx^{k+1}_{i^k}$, with $i^k$   in the set $\mathcal N\triangleq \{1,\ldots, N\}$, thus generating the  vector $\bx^{k+1}$.
When updating   block $i^k$, in general, the worker does not have   access to the current vector $\bx^k$, but it will   use instead the
  local estimate
$\bx^{k-\mathbf{d}^k} \triangleq (x_1^{k-d_1^k}, x_2^{k-d^k_2}, \ldots, x_N^{k-d^k_N})$, where  $\mathbf{d}^k \triangleq (d_1^k, d_2^k, \ldots, d_N^k)$ is  the
``vector of
delays'', whose components $d_i^k$ are
 nonnegative integers. Note that $\bx^{k-
\mathbf{d}^k}$ is
nothing else but a   combination  of delayed,
  block-variables.
The way each worker forms its own estimate $\bx^{k-\mathbf{d}^k}$ depends on the particular
  architecture under consideration
  and it is
immaterial to the  {analysis} of the algorithm.
We only observe here  that if all delays ${d}^k_i$ are zeros, the model reduces to a standard synchronous one. \newline
Given $\bx^{k-\mathbf{d}^k}$ and $i^k$,   block $\bx^k_{i^k}$ is updated
by  solving the following   {\em strongly
convex} block-approximation of  Problem~\eqref{eq:optimization_problem}: \vspace{-0.1cm}
\begin{equation}\label{eq:local model with ik}
\hat \bx_{i^k}(\bx^{k-\mathbf{d}^k})
\triangleq \argmin_{\bx_{i^k} \in \tilde{\mathcal{X}}_{i^k}(\bx^{k-\mathbf{d}^k})  }
 \tilde f_{i^k}(\bx_{i^k};
{\bx}^{k-\mathbf{d}^k}) + g_{i^k}(\bx_{i^k}),\vspace{-0.1cm}
\end{equation}
and then setting\vspace{-0.3cm}
\begin{equation}
\bx^{k+1}_{i^k} = \bx^{k}_{i^k}  + \gamma\, \left(\hat \bx_{i^k}(\bx^{k-\mathbf{d}^k})  -
\bx^{k}_{i^k}\right).
\label{update} \vspace{-0.1cm}
\end{equation}
In \eqref{eq:local model with ik},  $\tilde f_{i^k}$ and $\tilde{\mathcal{X}}_{i^k}$   represent a
strongly convex surrogate of $f$  and   a convex set obtained replacing the nonconvex functions
defining  $\mathcal{X}_{i^k}$  by suitably chosen  upper convex approximations, respectively;
both $\tilde f_{i^k}$ and $\tilde{\mathcal{X}}_{i^k}$  are built using the out-of-sync information $
\bx^{k-\mathbf{d}^k}$.  {If the set $\mathcal{X}_{i^k}$ is convex, then we will
always take $\tilde{\mathcal{X}}_{i^k}=\mathcal{X}_{i^k}$.}
In  (\ref{update}), $\gamma\in (0,1]$ is the stepsize. Note that, in the above asynchronous
model, the worker that is in charge of the computation \eqref{eq:local model with ik} and the
consequent update (\ref{update}) is immaterial.  
\smallskip

\noindent \textbf{Major contributions:}
 {Our main contributions are:}
\smallskip

\noindent 1. {\em A new probabilistic model for asynchrony fixing some unresolved issues}:
Almost all modern asynchronous algorithms  for convex and nonconvex problems are modeled in
a probabilistic way.
We put forth a novel probabilistic model describing the statistics of $(i^k,\mathbf{d}^k)$  that differs markedly from existing ones.
This new model allows us not only to fix some important theoretical issues that mar most of the papers
in the field (see discussion below on related work),  but it also lets us   analyze for the first time in a
sound way   several  {practically used and effective} computing settings and  new
asynchronous  algorithms. For instance,   it is
widely accepted that  in shared-memory   systems, the best performance are  obtained by first
partitioning the  variables among cores, and then letting each core  update in an asynchronous fashion their own
block-variables, according to some randomized cyclic rule.
To the best of our knowledge,  this is the first work   proving convergence of such
practically effective methods in an asynchronous setting.
\smallskip

\noindent 2. {\em The ability to effectively deal with nonconvex constraints}: All the works in the literature but \cite{davis2016asynchronous,DavisEdmundsUdell} can deal only with unconstrained or convex constrained problems.
On the other hand, the algorithms in \cite{davis2016asynchronous,DavisEdmundsUdell}   require at each iteration the computation of the global optimal solution of nonconvex subproblems, which, except in few special cases, can be as difficult as solving the original nonconvex problem.
Our method is the first asynchronous method that allows one to deal (under adequate assumptions)
with nonconvex constraints while solving only  strongly convex subproblems.\smallskip

\noindent 3. {\em The possibility to leverage potentially complex, but effective subproblems
\eqref{eq:local model with ik}}:
 Asynchronous methods so far are all built around a proximal linearization method, which
corresponds, in our framework, to setting
\[
\tilde f_{i}(\bx_{i^k};{\bx}^{k-\mathbf{d}^k}) =  \nabla_{\bx_{i^k}} f\left({\bx}^{k-\mathbf{d}^k}\right)^\text{T}
\left(\bx_{i^k} -
\bx_{i^k}^{k-d_{i^k}^k}\right) + \beta \,\| \bx_{i^k} - \bx_{i^k}^{k-d_{i^k}^k}\|^2_2,
\]
for some   constant $\beta>0$.
This choice often leads to   efficient solution methods and, in some cases,
even to subproblems that admit  a solution in closed-form. For instance,  it has been shown to be very efficient on  composite
 quadratic problems, like  LASSO. However, moving to more nonlinear
problems,  one may want to use more complex/higher order models. In fact,  the more sophisticated the subproblem \eqref{eq:local model with ik}, the better the
overall  behavior of
the algorithm (at least in terms of iterations) is. This happens at the price of computationally more expensive
subproblems. But in asynchronous and distributed methods,    the bottleneck
is often given by the communication cost. In
these cases, it might be desirable  to reduce the communication overhead at the price of  more complex subproblems to solve.  Furthermore,   there are   many application for which one can define  subproblems
 that, while not being proximal linearizations, still admit closed-form solutions (see. e.g., \cite{scutari2014decomposition,daneshmand2015hybrid}). Overall, the ability to use more complex subproblems is an additional degree of freedom
that one may want to exploit  to improve the performance of the algorithm.\vspace{0.1cm}  \newline
\noindent 4. {\em Almost sure convergence and complexity analysis}: We prove i) almost sure convergence  to stationary solutions of   Problem \eqref{eq:optimization_problem}; and ii) convergence to  $\epsilon$-stationary solutions in an $\mathcal O(\epsilon^{-1})$  number of iterations.
We remark that our  convergence results match  similar ones in the literature \cite{liu2015asynchronous,liu2015asyspcd,davis2016asynchronous,DavisEdmundsUdell,peng2016arock}, which however were obtained in a simplified setting (e.g., only for unconstrained or convex constrained problems) and under  unrealistic probabilistic assumptions on the  pair index-delay $(i^k,\mathbf{d}^k)$ (see discussion on related work). Our analysis builds  on an induction technique based on   our  probabilistic model and a novel Lyapunov function that properly combines variable dynamics and their delayed versions. 
 \vspace{0.01cm} \newline
\noindent 5. {\em A theoretical almost linear speed-up for a wide range of number of cores}: The holy grail of asynchronous methods is the ideal linear   speed-up (with respect to the number of workers). This theoretical limit is  not achievable in practice; in fact,    as the number of workers increases, the effective speedup is always limited by  associated overheads (communication costs, conflicts, etc.), which  make the linear  growth  impossible to achieve for arbitrarily large   number of workers.  
By using  the number of iterations needed to achieve an $\epsilon$-stationary solution as a proxy for the computational time and leveraging  our new Lyapunov function, we are able to show     almost linear speed-up
in many settings of practical interest.
This is the first   theoretical result on speedup, based on a realistic probabilistic model for asynchrony (see discussion in  contribution 1). 

\noindent
{\bf Related work.} Although asynchronous block-methods have a long history (see, e.g.,
\cite{baudet1978asynchronous,chazan1969chaotic, Bertsekas_Book-Parallel-Comp, frommer2000asynchronous,tseng1991rate}), their revival and  probabilistic analysis
have taken place only  in recent years; this is mainly  due  to the
current trend towards huge scale
optimization and the availability
of  ever more complex computational architectures  that call for efficient and resilient  algorithms.
Indeed, asynchronous parallelism has been   applied to  many state-of-the-art optimization algorithms,
including stochastic gradient methods
\cite{nedic2001distributed,Hogwild!,lian2015asynchronous,huo2016asynchronous,
	Mania_et_al_stochastic_asy2016,leblond2017Asaga,Pedregosa2017breaking} and  ADMM-like schemes
\cite{hong2014distributed,wei20131,iutzeler2013asynchronous}.
Block-Coordinate Descent (BCD) methods  are part of the folklore in optimization;
more recently,   they
have been proven to be  particularly effective in solving
very large-scale problems arising, e.g.,  from  data-intensive applications. Their asynchronous counterpart has been introduced and studied in the seminal work
\cite{liu2015asyspcd}, which  motivated and oriented much of subsequent research in the field, see e.g.
\cite{liu2015asynchronous,davis2016asynchronous,DavisEdmundsUdell,peng2016arock,peng2016convergence}.
We refer the interested reader to  
{\cite{Wright2015}  and references therein for a detailed   overview of BCD methods.
	There are several differences between the above methods and the framework proposed in this
	paper, as detailed next.
\newline\noindent $\bullet$ \textit{On the probabilistic model:} All current   probabilistic  models for asynchronous BCD methods are
based on the (implicit or explicit) assumption that the random variables $i^k$ and $\mathbf{d}^k$  are {\em independent};
this greatly simplifies the convergence analysis.
However, in reality there is a strong dependence of the delays  $\mathbf{d}^k$ on the updated block $i^k$.  Consider the setting where the variables are  partitioned among two workers and each worker updates only its own block-variables; let $\mathcal N_1$ and $\mathcal N_2$ be the index set of  the  blocks controlled by worker $1$ and $2$, respectively.  It is clear that  in the updates of worker $1$ it will always be   $d_i^k=0$, for all $i\in \mathcal N_1$ and $k$, while (at least some) delays $d_i^k$  associated with the blocks $i\in \mathcal N_2$ will be   positive; the opposite happens to worker two. The independence assumption is unrealistic also in settings where all the workers share all the variables. Blocks that are updated less frequently than others, when updated, will have larger associated delays. This happens, for instance, in problems where i) some blocks are more expensive to update than others, because they are larger,  bear more nonzero entries, or  data retrieval requires longer times;   or ii) the updates are carried by  heterogeneous workers (e.g., some are faster or busier than others).
 We tested this assumption, performing an asynchronous algorithm on two different architectures and measuring the average delay corresponding to different blocks updated. The experiments were performed on a shared-memory system with 10 cores of an Intel E5-2699Av4 processor. An asynchronous algorithm was applied to a LASSO problem  {\cite{tibshirani1996regression}} with 10000 variables, partitioned uniformly into 100 contiguous blocks; the Hessian matrix was generated with high sparsity on several rows. 
 All the cores can update any block, selected  uniformly at random.  {We found that}  blocks associated with the sparse rows of the Hessian have delays $\mathbf{d}^k$ with components between 0 and 3, while the delays of the other  blocks  were all bigger than 20. 
 Even when the computing environment is homogeneous and/or the block updates have the same cost, the aforementioned dependence persists.  We simulated a message-passing system on Purdue  Community Cluster Snyder; we used two nodes of the cluster, each of them equipped with  10 cores of an Intel Xeon-E5 processors and its own shared memory. Every node can update every block, selected uniformly at random. We ran an asynchronous algorithm on the  same LASSO problem described above but now with a dense Hessian matrix. The blocks updated by node 1 have an average delay of 12 while those updated by node 2  experience an average delay of 22.
This can be due to several uncontrollable factors, like operation system and memory schedulers, buses controllers, etc.,  which are hard to rigorously model and analyze.


Another   unrealistic assumption often made in the literature \cite{Hogwild!,liu2015asynchronous,liu2015asyspcd,davis2016asynchronous}
is that the block-indices $i^k$  are selected \emph{uniformly} at random. While this assumption   simplifies the
convergence analysis, it limits the applicability of the model; see Examples 4 and 5 in Section~\ref{sec:Examples}.
 In a nutshell, this assumption may be
satisfied  only if all workers have
the same computational power and have access to all variables.

We conclude the discussion on probabilistic models underlying asynchronous
algorithms mentioning the line of work dealing with stochastic gradient 
methods. Stochastic gradient methods are similar to block-descent approaches in that at each
iteration sampling is performed to determine the nature of the update, but sampling is
done among functions in an optimization problem minimizing the sum of functions,
as opposed to block variables. 
A related, albeit different, issue of
independence in the probabilistic models used in stochastic gradient methods was first noted in the technical 
report \cite{Mania_et_al_stochastic_asy2016}, see also \cite{leblond2017Asaga,Pedregosa2017breaking} for
further developments. These papers circumvent the issue by enforcing independence 
(a) using a particular manner of 
labeling iterations as well as 
(b) reading the entire vector of variables regardless of the sparsity 
pattern among the summand functions in the objective.
However, the analysis in \cite{Mania_et_al_stochastic_asy2016,leblond2017Asaga,Pedregosa2017breaking} is (c) only performed in the context of strongly
convex unconstrained problems, (d) involves uniform sampling and (e) is only applicable
for the shared memory setting.
Thus, while the analysis and procedures described in the references above are interesting,
on the whole requirements (b)-(e) make these proposals  
of marginal interest in the context of block-descent methods 
(even assuming they can actually be adapted to our setting).
\newline\indent Differently from the aforementioned works,  our more general and sophisticated  probabilistic model neither postulates the independence between $i^k$ and $\mathbf{d}^k$ nor  requires  artificial changes in the algorithm [e.g., costly unnecessary readings, as in (b)] to enforce it; it    handles instead    the potential dependency among variables directly. By doing so, one can establish  convergence  without
  requiring any of the restrictive conditions (b)-(e), and  significantly enlarge the class of computational architecture falling within the model [e.g., going beyond  (d) and (e)]$-$see  Section~\ref{sec:Examples} for several examples.
{The necessity of
 a new    probabilistic model  of asynchrony in BCD methods was first observed in our conference works \cite{cannelli2016parallel,cannelli2017asynchronous} while  the foundations of our approach  were
 presented  in our technical reports \cite{companion2,companion} along with some numerical results. %
 Here we improve  the  analysis of \cite{companion2,companion} by   relaxing
 considerably the assumptions for convergence  and tightening  the complexity bounds. }\newline 
 \noindent $\bullet$ \textit{Nonconvex constraints}: Another important feature of our algorithm is the ability to handle nonconvex objective functions
 and  nonconvex constraints by an algorithm that only needs to solve, at each iteration, a strongly convex optimization subproblem.  Almost  all  asynchronous  methods cited above can handle only convex
 optimization problems or, in
 the case of fixed point problems, nonexpansive mappings.
 The exceptions are  \cite{lian2015asynchronous,yun2014nomad} and, more relevant to our  setting,
 \cite{davis2016asynchronous,DavisEdmundsUdell} that
 study
 unconstrained and constrained  nonconvex optimization problems, respectively.  However,  the papers dealing with constrained problems, i.e.
 \cite{davis2016asynchronous,DavisEdmundsUdell}, propose algorithms that require, at each iteration, the global
 solution of nonconvex subproblems.
 Except for few  cases,
 the subproblems could be   hard to solve and potentially as difficult as the original one.
 \newline
 \noindent $\bullet$ \textit{Successive Convex Approximation}:
 All the asynchronous algorithms described so far use proximal linearization to define
 subproblems. As   already pointed out, this is the first paper where subproblem  models able to capture more structure of the objective functions are
 considered. This offers  more freedom and flexibility to tailor the minimization algorithm to the problem
 structure, in order to obtain more efficient solution methods.
 \newline \textbf{Notation}: We use the following notation for  random variables and their realizations: underlined
 symbols denote
 random variables, e.g., $\underline{\mathbf{x}}^k$, $\mathbf{x}^{k-\underline{\mathbf{d}}^k}$, whereas the
 same symbols with no underline
 are the corresponding realizations.\vspace{-0.6cm}

\section{Asynchronous Algorithmic Framework}
\label{sec:model}\vspace{-0.2cm}
In this section we introduce the assumptions on Problem \eqref{eq:optimization_problem} along with the formal description of the proposed algorithm. For simplicity of presentation,  we begin studying   \eqref{eq:optimization_problem}
 assuming that there are only convex  constraints, i.e.,  all ${\cal X}_i$   are convex. 
 This unnecessary
 assumption will  be removed in Section \ref{nonconvex}.\smallskip


\noindent \textbf{Assumption A (On Problem  (\ref{eq:optimization_problem})).}\vspace{-0.1cm}
\begin{description}
	\item[(A1)]  Each set $\mathcal{X}_i\subseteq\mathbb{R}^{n_i}$ is nonempty, closed, and convex;\smallskip
	\item[(A2)] $f:\mathcal{O}\rightarrow\mathbb{R}$ is $C^1$, where $\mathcal{O}$ is an open set containing $\mathcal{X}\triangleq  \mathcal{X}_1\times \cdots \times \mathcal{X}_N$;\smallskip
	\item[(A3)]  $\nabla_{\mathbf{x}_i} f$ is   $L_{f}$-Lipschitz continuous on $\mathcal{X}$; 
	\smallskip
	\item[(A4)] Each $g_i:\mathcal{O}_i\rightarrow\mathbb{R}$ is convex, possibly nonsmooth, and $L_g$-Lipschitz continuous on $\mathcal{X}_i$, where $\mathcal{O}_i$ is an open set containing $\mathcal{X}_i$;
	\smallskip
	\item[(A5)] $F$ is coercive on $\mathcal{X}$, i.e., $\underset{ \mathbf{x} \in \mathcal{X},  \|\mathbf{x}\| \to \infty}{\lim}F(\mathbf{x})=+\infty$.
	
\end{description}

\noindent 
These assumptions are rather standard. For example, A3 holds trivially if $\mathcal{X}$ is bounded and $\nabla f$ is locally Lipschitz. We remark that in most practical cases the
$g_i$'s are  norms or polyhedral functions and A4 is readily satisfied. Finally, A5 guarantees the existence of a solution.

We  introduce now our algorithmic asynchronous  framework. The asynchronous iterations performed by the workers are given in \eqref{eq:local model with ik} and \eqref{update} [cf.~Section \ref{sec:intro}]. However, the analysis of the algorithm based directly on \eqref{eq:local model with ik}-\eqref{update} is  not a simple task. 
The key idea   is then
to introduce a ``global view'' of \eqref{eq:local model with ik}-\eqref{update} 
that captures through a unified, general, probabilistic model several specific computational architectures/systems and asynchronous modus operandi.
The iteration $k\to k+1$ is triggered when a block-component $i^k$   of the current $\bx^k$ is updated by some worker using (possibly) delayed information $\bx^{k-\d^k}$, thus  generating the new vector $\bx^{k+1}$. Note that, in this model, the worker that  performs the update is immaterial. Given  \eqref{eq:local model with ik} and \eqref{update}, it is clear that  the update  $\bx^k\to \bx^{k+1}$ is fully determined once $i^k$ and $\d^k$ are specified. In several asynchronous methods, the index $i^k$ is chosen randomly. Even when this is not the case,  the   values of $i^k$ and $\d^k$ are difficult to preview  beforehand, because they  depend on several factors which are hard to  model mathematically, such as the computational architecture, the specific hardware, the communication protocol employed by the workers, possible hardware failures, etc..
   Therefore, we model the sequence  of pairs $\{(i^k, \d^k)\}_{k\in \mathbb N_+}$ generated by the algorithmic process as a realization of a stochastic process; the probabilistic space associated to this stochastic process is formally introduced in Section \ref{sec:probabilistic}.    The proposed general asynchronous model is summarized in
Algorithm~\ref{alg:global}, which we term Asynchronous FLexible ParallEl Algorithm (AsyFLEXA).

  \begin{algorithm}[t]
	\caption{Asynchronous FLexible ParallEl Algorithm (AsyFLEXA)}
	\label{alg:global}
	\begin{algorithmic}
		\STATE{\textbf{Initialization:}} $k=0$, $\mathbf{x}^0\in\mathcal{X}$,  $\gamma\in(0;1]$.
		\WHILE{a termination criterion is not met}
		\STATE{\texttt{(S.1)}  The random variable $(\underline{i}^k,\underline{\mathbf{d}}^k)$ is realized as  $(i^k,
		\mathbf{d}^k)$;}
		\STATE{\texttt{(S.2)}}  $\hat{\mathbf{x}}_{i^k}(\mathbf{x}^{k-\mathbf{d}^k})$ is computed: \vspace{-0.2cm}
\begin{equation}\label{eq:best-response-convex}
\hat{\bx}_{i^k}(\bx^{k-\mathbf{d}^k})
\triangleq \argmin_{\bx_{i^k} \in {\mathcal{X}}_{i^k}}
 \tilde f_{i^k}(\bx_{i^k};\bx^{k-\mathbf{d}^k}) + g_{i^k}(\bx_{i^k}),\vspace{-0.2cm}
\end{equation}
		\STATE{\texttt{(S.3)}}     $\mathbf{x}^k_{i^k}$ is acquired;
\STATE{\texttt{(S.4)}  The block $i^k$ is updated:\vspace{-0.2cm}
\begin{equation}\label{update-Algo1}\mathbf{x}_i^{k+1}=\begin{cases}\mathbf{x}^k_i+\gamma(\hat{\mathbf{x}}_i
(\mathbf{x}^{k-\mathbf{d}^k})-\mathbf{x}_i^k),&\text{if }i=i^k\\
\mathbf{x}_i^k&\text{if }i\neq i^k\end{cases}\vspace{-0.2cm}\end{equation}}
		\STATE{\texttt{(S.5)}   $k \leftarrow k+1;$} 	\ENDWHILE
	\end{algorithmic}
\end{algorithm}

\noindent \textbf{Discussion on Algorithm~\ref{alg:global}.} Several comments are in order.
\begin{asparaenum} \item \emph{On the generality of the model:} Algorithm~\ref{alg:global} represents a gamut of asynchronous schemes and architectures, all captured  in an abstract and  unified way by the stochastic process modeling the specific mechanism of generation of the delay vectors $\mathbf{d}^k$ and  indices $i^k$ of the blocks to updates.
 For concreteness, we show next how Algorithm~\ref{alg:global} customizes when  modeling  asynchrony in shared-memory and message passing-based  architectures. \smallskip
	
	\noindent\textit{Example 1: Shared-memory systems.} Consider a shared-memory system wherein multiple cores update in an asynchronous fashion blocks of the  vector $\bx$, stored in a shared memory. An iteration $k \to k+1$ of Algorithm~\ref{alg:global} is triggered when a core  writes  the (block) update  $\bx_{i^k}^{k+1}$ in the shared memory (Step 4). Note that the cores need not know the global iteration index $k$.  No memory lock is assumed, implying that   components
of the variables may be written by some cores while other components are  simultaneously
read by others. This \emph{inconsistent read}   produces vectors  $\bx^{k-\mathbf{d}^k} = (x_1^{k-d_1^k}, x_2^{k-d^k_2}, \ldots, x_N^{k-d^k_N})$,   to be  used in the computation of $\hat{\mathbf{x}}_{i^k}$ (Step 2), whose (block) component  $\bx_i^{k-d_i^k}$ is a (possibly) delayed version  of block $i$ read by the core  that is going to perform the update.
	Note that, while $\bx_i^{k-d_i^k}$  existed in the shared memory at some point in time, the entire delayed vector $\bx^{k-\mathbf{d}^k}$ might have not at any time. Also, in Step 4, it is tacitly assumed that the update of a block is {\em atomic} (the block is written in the shared memory as a whole) and while a core is writing that block no other core can modify the same block. This is minor requirement, which can be easily enforced in modern architectures either by a block-coordinate look or using a   dual-memory writing approach, see   \cite[Section 1.2.1]{peng2016arock}.

\begin{figure}
\begin{center}
\includegraphics[width=11cm]{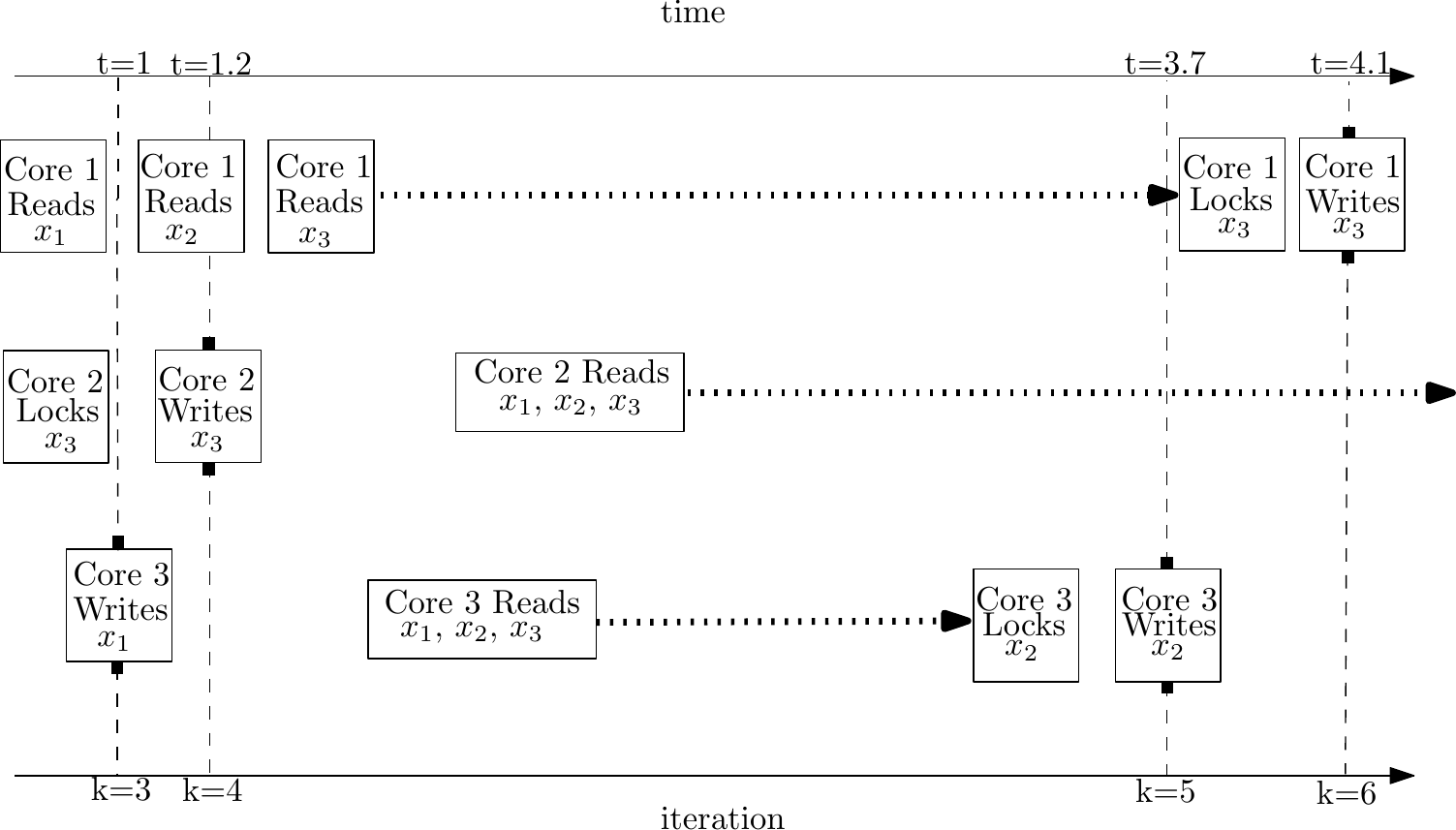}
\end{center}\vspace{-0.2cm}
\caption{AsyFLEXA modeling block asynchronous updates in a shared-memory system: three cores, vector variables $\bx\in\mathbb{R}^3$, scalar  blocks ($n_i=1$, for all $i$).
}
\label{fig:inconsistent_read}\vspace{-0.4cm}
\end{figure}
Figure~\ref{fig:inconsistent_read} shows few iterations of the  algorithm dynamics in the asynchronous setting described above. The (continuous) time when operations (reading, writing, computation) are performed is indicated in the top  horizontal axes whereas the global (discrete) iteration counter is reported in the bottom axes. 
The asynchronous updates happen as follows. At iteration $k=3$, Core $3$ writes $x_1^3$; therefore,   $\bx^3$ differs from $\bx^2$
in the first component. Core $2$ locks $x_3^2$ to quickly read it and perform the linear combination with $\hat x_3(\bx^{3-\mathbf{d}^3})$ [cf. (\ref{update-Algo1})], and updates $x_3$; therefore $\bx^4$ differs from $\bx^3$ in just the 3rd component. Note that core $2$ reads $x^2_3$ which is equal to $x^3_3$, so $d_3^3=0$; this is because between the lock and the writing of core $2$, no other cores wrote $x_3$ (core $3$ updates $x_1$). At iteration $k=5$, core $3$ writes $x_2^5$. In this case $\bx^{4-\mathbf{d}^4}$, used to compute $x^5_2 = (1-\gamma)x^4_2+\gamma\hat x_2(\bx^{4-\mathbf{d}^4})$, is exactly equal to $\bx^4$, since core $3$ reads the vector entirely after the last update, so $\mathbf{d}^4=\mathbf{0}$. A different situation happens when iteration $k=6$ is  triggered by core $1$: the vector used by the core to perform its update is   ${\bx}^{5-\mathbf{d}^5}$, which is such that
$x^{5-\mathbf{d}^5}_1 = x^2_1 \neq x^3_1$, $x^{5-\mathbf{d}^5}_2 = x^2_2$, and $x^{5-\mathbf{d}^5}_3 = x^4_3\neq x^2_3$; therefore, $\bx^{5-\mathbf{d}^5}$ never existed in the shared memory at any time. It can be seen that the vector of delays at $k=6$ reads $\d^5 = (3,1,0)^{\text{T}}$. 
 Note that the delay vector $\mathbf{d}^k$ used at a given iteration may not be unique: different values for the components $d_i^k$ may produce the same delayed vector $\mathbf{x}^{k-\mathbf{d}^k}$. For instance, in the example above, the  vector $(2,1,0)^{\text{T}}$ could have been  used  in place of $\mathbf{d}^5$. 
\smallskip
	
	\noindent\textit{Example 2: Message passing  systems.}   Consider a message passing-based system: multiple computational units (e.g., clouds, cluster computers) are connected through a (directed) graph, modeling the communication pattern among the units. Suppose that every worker has in charge the update of a set of block variables, partitioned  among all the workers. In this setting, Algorithm~\ref{alg:global} still models asynchronous updates and communications.    There is no shared memory;   every worker updates its own variables writing  its own local memory, and then broadcasts its   updates to its neighbors, according to a given protocol, which specifies, for instance, how ofter the communications will happen and with whom. In this setting,  $\bx^{k-\d^k}$   corresponds to the most recent information a worker has received from the others at the time of its update.\smallskip

\item \textit{On the surrogate functions $\tilde f_{i}$.}
A degree of freedom offered by the proposed framework is the choice of the surrogate function used in the subproblems (\ref{eq:best-response-convex}) solved by the workers at each iteration.   We consider the following general class of surrogate functions (we denote by $\nabla \tilde f_{i}$
the partial gradient of $\tilde f_i$ with respect to the first argument).

	\noindent \textbf{Assumption B (On the surrogate functions  $\tilde{f}_i$'s).} Given $\tilde f_{i}: \mathcal{X}_i\times \mathcal{X}\to \mathbb{R}$, with $i\in\mathcal{N}$, we assume:
	\begin{description}
		\item[ (B1)] $\tilde f_{i} (\bullet; \mathbf{y})$ is $C^1$ on an open set containing $\mathcal{X}_i$, and $c_{\tilde{f}}$-strongly
		convex
		on $ \mathcal{X}_i$, for all $\mathbf{y}\in \mathcal{X}$;
		\smallskip
		\item[  (B2)]   $\nabla \tilde f_{i} (\mathbf{y}_i;\mathbf{y}) = \nabla_{\mathbf{y}_i} f(\mathbf{y})$, for all $\mathbf{y} \in  \mathcal{X}$;
		\smallskip
		\item[  (B3)]  $\nabla \tilde f_{i} (\mathbf{y}_i;\bullet)$
		is $L_B$-Lipschitz continuous
		on $ \mathcal{X}$,
		for all $\mathbf{y}_i \in  \mathcal{X}_i$;\smallskip
		\item [  (B4)]$\nabla \tilde f_{i} (\bullet;\mathbf{y})$
		is $L_E$-Lipschitz continuous
		on $ \mathcal{X}_i$, for all $\mathbf{y} \in  \mathcal{X}$. 
	\end{description}\smallskip

\noindent The surrogate  $\tilde f_i(\bullet;\mathbf{x}^k)$  should be regarded as  a (simple) strongly convex
local approximation of $f$ around $\mathbf{x}^k\in \mathcal X$, that
preserves the first order properties of $f$. 
Finding a surrogate $\tilde f_i$ that satisfies Assumption B is in general non difficult; in
any case, one can  always  choose $\tilde f_{i} (\mathbf{x}_i; \mathbf{x}^k)
= \nabla_{{\mathbf{x}}_i}   f(\mathbf{x}^k)^\text{T}(\bx_i -\mathbf{x}_i^k) + \beta\| \bx_i - \mathbf{x}_i^k\|^2_2$, where $\beta$ is a
positive constant, which leads to  the classical proximal-gradient update. However, having the possibility to use a different $\tilde f_i$ may be
useful to exploit some potential structure in the problem; of course,    a trade-off is expected: the more complex the
$\tilde f_i$, the more information will be retained  in $\hat \bx_i(\mathbf{x}^k) $, but also the more intensive its computation  is
expected  to be. On the other hand, the solution of more complex subproblems will in general decrease the number of
information exchanges in the system, which may be a key advantage in many applications.
Some valid instances of  $\tilde f_i$'s going  beyond the proximal-gradient choice are discussed next; we refer the interested reader to \cite{FLEXA,facchinei2016feasible} for further examples.
\smallskip

	\noindent $\bullet$ If $f(\mathbf{x}_1,\ldots,\mathbf{x}_N)$ is block-wise uniformly convex, instead of linearizing $f$
	one can exploit a second-order approximation and set  $\tilde{f}_i(\mathbf{x}_i;
	\mathbf{x}^k)=f(\mathbf{x}^k)+\nabla_{\mathbf{x}_i}f(\mathbf{x}^k)^\text{T}(\mathbf{x}_i-\mathbf{x}^k_i)+\frac{1}{2}
	(\mathbf{x}_i-\mathbf{x}^k_i)^\text{T}\nabla^2_{\mathbf{x}_i\mathbf{x}_i}f(\mathbf{x}^k)
	(\mathbf{x}_i-\mathbf{x}^k_i)+ \beta\|\mathbf{x}_i-\mathbf{x}_i^k\|^2_2$;\smallskip
	
	\noindent$\bullet$ In the same setting as above, one can also better preserve the partial convexity of $f$ and set
	$\tilde{f}_i(\mathbf{x}_i;\mathbf{x}^k)=f(\mathbf{x}_i,\mathbf{x}_{-i}^k)+ \beta\|\mathbf{x}_i-\mathbf{x}_i^k\|^2_2$,
	where $\mathbf{x}_{-i}\triangleq (\mathbf{x}_{1},\ldots, \mathbf{x}_{i-1},
	\mathbf{x}_{i+1},\ldots ,\mathbf{x}_{N})$;\smallskip
	
	\noindent$\bullet$ As a last example, suppose that $f$ is the difference of two convex functions $f^{(1)}$ and
	$f^{(2)}$, i.e., $f(\mathbf{x})=f^{1}(\mathbf{x})-f^{2}(\mathbf{x})$, one can preserve the partial convexity in $f$ setting
	$\tilde{f}_i(\mathbf{x}_i;\mathbf{x}^k)=f^{(1)}(\mathbf{x}_i, \mathbf{x}_{-
	i}^k)-\nabla_{\mathbf{x}_i}f^2(\mathbf{x}^k)^\text{T}(\mathbf{x}_i-\mathbf{x}^k_i)+{ \beta}\,\|
	\mathbf{x}_i-\mathbf{x}_i^k\|^2_2$.\vspace{-0.4cm}
   \end{asparaenum}

\section{AsyFLEXA: Probabilistic Model} \label{sec:probabilistic}\vspace{-0.2cm}
In this section, we complete the description of AsyFLEXA, introducing  the probabilistic model underlying the generation of the pairs index-delays. 

 Given   Problem~\eqref{eq:optimization_problem} and an initial point $\bx^0$,   the pair  $(i^k,\d^k)$ in Step 1 of Algorithm~\ref{alg:global}, for each $k$,  is a realization of a random vector $\underline{\boldsymbol{\omega}}^k\triangleq (\underline{i}^k,
\underline{\mathbf{d}}^k)$, taking values on $\mathcal N \times \mathcal D$,
where ${\cal D}$ is the set    of all possible delay vectors.
We anticipate that all the delays $d_i^k$ are assumed to be bounded (Assumption C below), i.e., $d^k_i \leq \delta$, for all $k$ and $i$. 
Hence,  ${\cal D}$ is contained in the set of all possible  $N$-length vectors whose components are
integers between $0$ and $\delta$. 
 Let   $\Omega$ be the sample space  of   all  the  sequences $\omega\triangleq\{ (i^k,
\mathbf{d}^k)\}_{k\in\mathbb{N}_+}$.\footnote{With a slight abuse of notation,
we denote by $\boldsymbol{\omega}_k$ the $k$-th element of the sequence $\omega\in \Omega$, and by
$\boldsymbol{\omega}^k$ the value taken by the random variable $\underline{\boldsymbol{\omega}}^k$ over $\omega$,
i.e.
$\underline{\boldsymbol{\omega}}^k(\omega) = \boldsymbol{\omega}^k$.}  We will use the following  
shorthand notation:  we set
 $\underline{\boldsymbol{\omega}}^{0:k} \triangleq (\underline{\boldsymbol{\omega}}^0,\underline{\boldsymbol{\omega}}^1, \ldots,
 \underline{\boldsymbol{\omega}}^k)$ (the first $k+1$ random variables);
 ${\boldsymbol{\omega}}^{0:k} \triangleq ({\boldsymbol{\omega}}^0,{\boldsymbol{\omega}}^1, \ldots,{\boldsymbol{\omega}}^k)$ ($k+1$ possible values
 for the random variables  $\underline{\boldsymbol{\omega}}^{0:k}$); and
 ${\boldsymbol{\omega}}_{0:k} \triangleq ({\boldsymbol{\omega}}_0,{\boldsymbol{\omega}}_1, \ldots,{\boldsymbol{\omega}}_k)$ (the first $k+1$
 elements of $\omega$).  We introduce next the probability space that will be used to build our probabilistic model.

 The sample space is $\Omega$. To define a $\sigma$-algebra on $\Omega$, we consider, for
  $k\geq 0$ and every $\boldsymbol{\omega}^{0:k}\in {\cal N}\times {\cal D}$,
 the  cylinder\vspace{-0.1cm}
 $$C^k(\boldsymbol{\omega}^{0:k}) \triangleq \{\omega\in \Omega: \boldsymbol{\omega}_{0:k} =
 \boldsymbol{\omega}^{0:k}
 \},\vspace{-0.1cm}$$ i.e., $C^k(\boldsymbol{\omega}^{0:k})$ is the subset of $\Omega$ of all sequences $\omega$ whose first $k$
 elements are  $\boldsymbol{\omega}^0, \ldots \boldsymbol{\omega}^k$. Let us
 denote by ${\cal C}^k$ the set of all possible $C^k(\boldsymbol{\omega}^{0:k})$  when $\boldsymbol{\omega}^t$, $t=0,
 \ldots, k$,  takes all possible values;
 note, for future reference, that ${\cal C}^k$ is a partition of $\Omega$. Denoting by $\sigma\left({\cal C}^k\right)$ the $\sigma$-algebra generated by  ${\cal C}^k$, define for all $k$, \vspace{-0.1cm}
\begin{equation}\label{eq:sub-sigma}
 {\cal F}^k\triangleq \sigma\left({\cal C}^k\right) \qquad \mbox{\rm and}\qquad
 {\cal F}\triangleq \sigma\left(\cup_{t=0}^\infty{\cal C}^t\right).
 \end{equation}
We have ${\cal F}^k\subseteq {\cal F}^{k+1} \subseteq {\cal F}$ for all $k$. The latter inclusion is obvious, the former
derives easily from the fact that any cylinder in ${\cal C}^{k-1}$ can be obtained as a finite union of cylinders in ${\cal
C}^k$.
\newline  The desired probability space  is fully defined once   $\mathbb{P}(C^k(\boldsymbol{\omega}^{0:k}))$, the probabilities
  of all  cylinders, are given. These probabilities should satisfy some very natural, minimal  consistency
  properties, namely:
  (i) the probabilities of the union of a finite number of disjoint cylinders should be equal to the sum of the probabilities
  assigned to each cylinder; and (ii) suppose that a cylinder $C^k(\boldsymbol{\omega}^{0:k})$ is contained in the union
  $U$ of a countably infinite number of other cylinders, then $\mathbb{P}(C^k(\boldsymbol{\omega}^{0:k})) \leq P(U)$.
  Suppose now that such a $\mathbb P$ is given. 
 Classical results
  (see, e.g., \cite[Theorem 1.53]{klenke2013probability})
  ensure that one can extend  these probabilities to a probability measure $P$ over $(\Omega, {\cal
  F})$, thus defining our working probability space $A \triangleq
  (\Omega, {\cal F}, P)$. By appropriately choosing the probabilities of the cylinders, we can model
  in a unified way many cases of practical interest; several examples are given in Section~\ref{sec:Examples}.

Given $A$, we can finally define    the discrete-time, discrete-value   stochastic process $\underline{\boldsymbol{\omega}}$,
where $\{\underline{\boldsymbol{\omega}}^k(\omega)\}_{k\in \mathbb{N}_+}$  is a sample path of the process. The
$k$-th entry
$\underline{\boldsymbol{\omega}}^k(\omega)$ of $\underline{\boldsymbol{\omega}}(\omega)-$the $k$-th element of
the sequence $\omega-$is a realization of the random vector  $\underline{\boldsymbol{\omega}}^k= (\underline{i}^k,
\underline{\mathbf{d}}^k):\Omega \mapsto \mathcal N\times \mathcal D$.
This process fully describes the evolution of Algorithm 1. Indeed,  given an instance of Problem (\ref{eq:optimization_problem}) and  a starting point,  the trajectories
of   the  variables $\bx^k$
and  $ \bx^{k-\mathbf{d}^k}$ are  completely determined once a sample path $\{ (i^k, \mathbf{d}^k)\}_{k\in
\mathbb{N}_+}$ is drawn from $\underline{\boldsymbol{\omega}}$.

Note that the joint  probability
$$
p_{\underline{\boldsymbol{\omega}}^{0:k}}(\boldsymbol{\omega}^{0:k})\triangleq \mathbb{P}
(\underline{\boldsymbol{\omega}}^{0:k}=
\boldsymbol{\omega}^{0:k})
$$
is simply the probability of the corresponding cylinder: $C^k(\boldsymbol{\omega}^{0:k})$.
 We will often need to consider the conditional probabilities  $p((i,\mathbf{d})\,|\, \boldsymbol{\omega}^{0:k})\triangleq
 \mathbb{P}(\underline{\boldsymbol{\omega}}^{k+1}=(i,\mathbf{d})|
 \underline{\boldsymbol{\omega}}^{0:k}={\boldsymbol{\omega}}^{0:k})$.
 Note that we have
 \begin{equation}\label{eq:conditional prob}
 p((i,\mathbf{d})\,|\, {\boldsymbol{\omega}}^{0:k}) = \frac{\mathbb{P}(C^{k+1}(\boldsymbol{\omega}^{0:k+1}))}
 {\mathbb{P}(C^{k}(\boldsymbol{\omega}^{0:k}))},
\end{equation}
 where   we tacitly assume $ p((i,\mathbf{d})\,|\, {\boldsymbol{\omega}}^{0:k})=0$, if $\mathbb{P}(C^{k}(\boldsymbol{\omega}^{0:k}))=0$.
We remark that these probabilities need not be  known in practice to implement the algorithm. They are instead determined 
based on
 the  particular system (hardware architecture,  software implementation,  asynchrony, etc.) one is interested to model. Here,  we make only some  minimal assumptions on these probabilities and stochastic model,  as stated next.


\noindent \textbf{Assumption C (On the probabilistic model).} Given Algorithm 1 and the stochastic process
$\underline{\boldsymbol{\omega}}$, suppose that\vspace{-0.2cm}
\begin{description}
	\item[(C1)] There   exists a $\delta\in \mathbb{N}_+$, such that $d^k_{i} \leq \delta$,  for all $i$ and $k$;\smallskip
	
	\item[(C2)]  For all $i\in\mathcal{N}$ and $\omega\in\Omega$, there exists at least one $t\in[0,\ldots,T]$, with $T>0$, such that $$\sum_{\mathbf{d}\in
		\mathcal D}p((i,\mathbf{d})\,|\, {\boldsymbol{\omega}}^{0:k+t-1})\geq p_{\min},\quad\text{if}\quad p_{\underline{\boldsymbol{\omega}}^{0:k+t-1}}(\boldsymbol{\omega}^{0:k+t-1})>0,$$ for some $p_{\min}>0$;\smallskip
	\item[(C3)] $d^k_{i^k} =0$,
	for any
 $k\geq0$.
\end{description}

\noindent
These are quite reasonable assumptions, with very intuitive interpretations. C1 just limits the age of the old
information used in the updates. Condition C2 guarantees that every $T$ iterations each block-index $i$ has a
non negligible positive probability to be updated. These are minimal requirements that are   satisfied in  {practically all computational environments.}
 The condition
 $d^k_{i^k} =0$  means that when a worker updates the $i^k$-th block,  it   uses the most
 recent value of that block-variable. This assumption
 is automatically  satisfied, e.g.,  in a message passing-based system  or in a shared memory-based architecture   if the
 variables are partitioned and
 assigned to different cores (see Example 6 in Section \ref{sec:Examples}).
 If instead all the cores can update all variables,
  $d^k_{i^k} =0$ can be   simply enforced by a
 software lock on the ${i^k}-$th block of the shared memory:  once a core $c$ has read a
 block-variable
 $\bx_{i^k}$,  no other core can  change it, until $c$ has performed its update. Note that   in practice it is very unlikely that  this lock
 affects
 the performance of the
 algorithm, since usually  the number of cores is  much smaller than the number
 of  block-variables. 
Actually, in some systems, this lock  can bring in some benefits.  For instance, consider two cores sharing all variables, with one core much faster than the other. A lock on $\bx_{i^k}^k$   will prevent    potentially  much older information to overwrite  most recent updates of the faster core.
Note also that conditions similar to C3 are required by all block asynchronous methods in the literature but \cite{peng2016arock}: they take the form of locking the variable to update  before performing a prox
operation 
\cite{liu2015asyspcd}. 

\begin{remark} \label{rem:key1} \rm
The knowledge of the probability space $A$ is by no means  required from the workers   to
perform the updates. One need not even specify explicitly the probability distribution; it is
sufficient to show that a probability space $A$ satisfying Assumption C exists for   the specific
system (e.g., computational architecture, asynchronous protocol, etc.)  under consideration. We
show next how to do so for several schemes of practical interest.  \vspace{-0.6cm}
\end{remark}

\subsection{Examples and special  cases}\label{sec:Examples} \vspace{-0.2cm} The proposed  model encompasses a gamut of practical  schemes, some of which are discussed next.
 {It is of note  that our framework allows us to analyze in a unified way not only  randomized methods, but also deterministic algorithms.}

\noindent
{\bf 1. Deterministic sequential cyclic BCD:} 
In a deterministic,   cyclic method there is only one core that cyclically updates all block-variables; for simplicity we assume  the natural order, from 1 to $N$. Since there is only one core, the reading is always consistent and there are no delays: $\cal D=\{\mathbf{0}\}$. To represent the cyclic choice it is now enough to assign probability 1 too all cylinders of the type\vspace{-0.2cm}
$$
 C^k = \{ \omega: \omega_0 = (1,\mathbf{0}), \omega_1 = (2,\mathbf{0}), \ldots, \omega_k = ((k \mod N) +1,\mathbf{0}) \}\vspace{-0.2cm}
$$
and  probability zero to all others. It is easy to see that Assumption C  is satisfied. This can be seen as a probabilistic model of the deterministic algorithm in \cite{scutari2014decomposition}. The consequence however is that, by  Theorem~\ref{Theorem_convergence}, convergence can be claimed only in a probabilistic sense (a.s.). 
This is not surprising, as we are describing a deterministic algorithm as limiting case of a probabilistic model.

\noindent
{\bf 2. Randomized sequential BCD:}
 Suppose now that there is only one core  selecting at each iteration randomly an index $i$, with a   positive probability.
 Therefore, at each
iteration, $\bx^{k-\mathbf{d}^k} = \bx^k$ or, equivalently, $\cal
D=\{\mathbf{0}\}$.  This  scheme can be described by a stochastic process,  where the cylinders are assigned arbitrary
probabilities but satisfying all the conditions given in  previous subsection. 

\noindent
{\bf 3. Randomized  parallel BCD:}
Suppose that there are $C$ cores and the
block-variables are partitioned in $C$ groups $I_1, I_2, \ldots,
I_C$; each set $I_c$ is assigned to one core only, say  $c$. 
Hence, if core $c$ performs the update at iteration $k$,   all variables $i\in I_c$ satisfy $d_i^k=0$.
Denote by $\boldsymbol{0}(c)$, $\boldsymbol{1}(c)$, \ldots, and
 $(\boldsymbol{C-1})(c)$,  $c=1,\ldots,C$, the $N$-length vectors whose components are zeros in the positions of the block-variables in the set $I_c$ and all $0, 1,\ldots, C-1$ in the other positions, respectively. Set  $\mathcal{D}=\{\boldsymbol{0}(c), \boldsymbol{1}(c), \ldots,
 (\boldsymbol{C-1})(c),\,\,c=1\ldots,C\}$, and denote by $c^k$ the core performing the update at iteration $k$. 
  Assign to the cylinders the following probabilities: $\forall i^0,i^1,\ldots,i^{2C-1},\ldots \in\mathcal{N}$,  \vspace{-0.1cm}
  \begin{align*}
 &\mathbb{P}(C^0((i^0,\boldsymbol{0}(c^0) ))) = 1/N,\smallskip\\
 & \mathbb{P}(C^1((i^0,\boldsymbol{0}(c^0) ), (i^1,\boldsymbol{1}(c^1) ))) = 1/N^2,\\
 \ldots\\
  &\mathbb{P}(C^{C-1}((i^0,\boldsymbol{0}(c^0) ), (i^1,\boldsymbol{1}(c^1) ),\ldots, (i^{C-1},\boldsymbol{C-1}(c^{C-1}) ))) = 1/N^C,\smallskip\\
  &\mathbb{P}(C^C(  (i^0,\boldsymbol{0}(c^0) ), (i^1,\boldsymbol{1}(c^1) ),\ldots, (i^{C-1},\boldsymbol{C-1}(c^{C-1}) ),
(i^C,\boldsymbol{0}(c^C) ))) = 1/N^{C+1},\smallskip\\
&\mathbb{P}(C^{C+1}  ( (i^0,\boldsymbol{0}(c^0) ), (i^1,\boldsymbol{1}(c^1) ),\ldots, (i^{C},\boldsymbol{0}(c^C) ),
(i^{C+1},\boldsymbol{1}(c^{C+1}) ))) = 1/N^{C+2},\smallskip\\
 &\ldots\smallskip\\
  &\mathbb{P}(C^{2C-1}
( (i^0,\boldsymbol{0}(c^0) ), (i^1,\boldsymbol{1}(c^1) ),\ldots,
(i^{2C-1},\boldsymbol{C-1}(c^{2C-1}) ))) = 1/N^{2C},\\
  &\ldots
\end{align*}
In words,
in the first $C$ iterations (from  $k=0$ to  $k=C-1$), all updates are performed using
the same vector $\bx^{k-\mathbf{d}^k} = \bx^0$; and at each iteration any index has uniform probability 
to be selected.
This situation is then repeated for the next $C$ iterations, this time using  $\bx^{k-\mathbf{d}^k} = \bx^C$, and so on.
This model clearly corresponds to a randomized parallel block-coordinate descent method wherein $C$ cores update $C$ block-variables chosen uniformly at random.   Note that Assumption C is trivially satisfied.

The example above clearly shows that defining probabilities by using the cylinders can be quite tedious even in simple
cases. Using \eqref{eq:conditional prob} we can equivalently define the probabilistic model by specifying the conditional
probabilities $p((i,\mathbf{d})\,|\, {\boldsymbol{\omega}}^{0:k})$, which is particularly convenient when at every  iteration
$k$ the probability that  $\underline{\boldsymbol{\omega}}^k$ takes  value $(i,\mathbf{d})$
is independent of $ \boldsymbol{\omega}^{0:k-1}$. We exemplify this alternative approach  in the following examples.

\noindent
{\bf 4. Asynchronous BCD  in shared memory systems:}
Consider a
generic shared memory system, under Assumption C3. Then, 
the set $\cal D$  is given by all the $N$-length vectors whose components are non negative integers between 0 and $\delta$. 
  Suppose that, at every $k$,  all cores select an index uniformly at random, but   the probabilities associated with the  delays can be  different. Then,  for every
  $k\geq 0
  $, given $\boldsymbol{\omega}^{0:k}$, and    $i\in {\cal N}$, we have
  \vspace{-0.1cm}
 $$
 \sum_{\d\in {\cal D}} p((i,\mathbf{d})\,|\, {\boldsymbol{\omega}}^{0:k}) = \frac{1}{N}.\vspace{-0.1cm}
 $$
  This setting is consistent with the one studied in \cite{lian2015asynchronous,davis2016asynchronous,DavisEdmundsUdell,liu2015asyspcd}.

  Our probabilistic model however   is more general than that of \cite{lian2015asynchronous,davis2016asynchronous,DavisEdmundsUdell,liu2015asyspcd}. For instance, differently from \cite{lian2015asynchronous,davis2016asynchronous,DavisEdmundsUdell,liu2015asyspcd}, we can easily model scenarios wherein  $
 \sum_{\d\in {\cal D}} p((i,\mathbf{d})\,|\, {\boldsymbol{\omega}}^{0:k})
 $ are not uniform and/or depend
 on the iteration and/or
  on the history of the algorithm. 
This
 possibility has important
 ramifications, since the assumption that the indices are selected uniformly at random is
 extremely strong and unrealistic. In fact, it is 
 satisfied only if all cores have the
 same computational power  and have access to all variables.
This is not the case, in most of  practical settings. For instance,  consider a computational architecture composed of two CPUs sharing all the variables, with one CPU  much faster than the other.
  If the recent history exhibits iterations with a small
value of $\|\d^k\|_\infty$, then it is more likely that the slower core will perform the next update, and vice versa. Similar situations are expected also in other common settings, such as shared memory systems with variable  partitioning (see Example 5 below) and message passing-based architectures.
This clearly shows that our model  captures realistic architectures more faithfully.

\noindent
{\bf 5. Asynchronous BCD  in shared memory systems with    variable
partitioning:} Consider the setting as in Example 4, but now partition the variables across cores, as   described in Example 3.
 This is the configuration most often used in numerical experiments, since it has  proven
to be most effective in practice;  it also models a message passing architecture.  
In order to satisfy C3, it is enough to set, for all $\boldsymbol{\omega}^{0:k}$ and $i\in I_c$,   \vspace{-0.1cm}
$$
p((i,\mathbf{d})\,|\, {\boldsymbol{\omega}}^{0:k}) = 0, \; \hbox{\rm if some $\d_j\neq \mathbf{0}$,  $j\in I_c$.}
$$

 A variant of this setting is the {\em without replacement} updating scheme considered in the  numerical experiments of \cite{liu2015asyspcd}:  the block-variables are partitioned among the cores and, at each ``epoch", variables in each partition are first randomly shuffled   and then updated cyclically by the core. This choice of the updates was shown to be numerically very effective. While \cite{liu2015asyspcd} cannot provide any theoretical analysis of such a scheme, we can easily cover this case by just  merging this example with  Example 2.

 \noindent \textbf{Other examples:} Several other examples can be considered, which we omit because of space limitation. Here we only mention that  it is quite straightforward to
  analyze by our model  also ``hybrid'' systems, which combine somehow two or more  examples described above.  {For instance,  consider   a cluster computer system
 wherein the optimization variables are partitioned across the machines;   let $I_m$ be the set of variables controlled by machine  $m$ and stored in its internal shared memory. The update of the variables in $I_m$ is performed by the processors/cores of machine $m$ according to some  shared memory-based asynchronous scheme (e.g.,   subject to inconsistent read). The  information on  the variables not in $I_m$ is instead  updated through communication with the other processors (message passing)}
\vspace{-0.4cm}

\section{AsyFLEXA: Convergence Results}\label{sec:convergence}\vspace{-0.2cm}
We present now our main convergence theorem, under Assumptions A-C. The extension to the  case of nonconvex constraints is addressed in Section \ref{nonconvex}.

	We will use $\|M_F(\mathbf{x})\|_2$ as a measure of optimality, with  \vspace{-0.1cm}
	\begin{equation}
	M_F(\mathbf{x})\triangleq \mathbf{x}-\underset{\mathbf{y}\in\mathcal{X}}{\text{arg min}}\left\{\nabla f(\mathbf{x})^\text{T}(\mathbf{y}-\mathbf{x})+g(\mathbf{y})+\frac{1}{2}\|\mathbf{y}-\mathbf{x}\|^2_2\right\}.\vspace{-0.1cm}
	\end{equation}
	This is a valid measure of stationarity because $M_F(\mathbf{x})$ is continuous and $\|M_F(\mathbf{x})\|_2={0}$ if
	and only if $\mathbf{x}$ is a stationary solution of Problem \eqref{eq:optimization_problem}.
	
	To state our major convergence result, we need to introduce first the following intermediate definitions.
Recalling the definition of
$T$ as in Assumption $C2$, let
 $\underline{\mathcal{K}}_i^k$   be the (random) set of iterations between $k-\delta$ and $k+T-1$ at which the block-variable $i$ has been updated,
$\mathcal{K}_i^k	\triangleq\{t\in[k-\delta;k+T-1]\,|\,\underline{i}^t=i\}$, while
$\bar{\underline{\mathcal{K}}}^k_i$ is the subset of $\underline{\mathcal{K}}_i^k$ containing  only the elements of $\underline{\mathcal{K}}_i^k$ (iterations) between $k-
\delta$ and $k-1$. Our convergence results leverage a  Lyapunov function $\tilde{F}$ that suitably combines present and past  iterates, and it is defined as
\begin{equation}
	\tilde{F}(\mathbf{x}^k,\dots,\mathbf{x}^{k-\delta})
	= F(\mathbf{x}^k)+\delta\frac{L_f}{2}\,\left(\sum\limits_{l=k-\delta}^{k-1}\left(l-(k-1)+\delta\right)\,\|\mathbf{x}^{l+1}-\mathbf{x}^l\|_2^2\right),
	\label{lyapunov}\end{equation}
	where it is understood that $\bx^l = \bx^0$, if  $l<0$; therefore,  $\tilde{F}$ is well defined for any $k\geq 0$. Note   that, by this convention, $\tilde{F}(\mathbf{x}^0,\dots,\mathbf{x}^{0-\delta})= F(\bx^0)$. Furthermore, 
we also have
$F^* \triangleq \underset{\bx \in {\cal X}}{\min}\; F(\bx)\leq\underset{\mathbf{x}^k,\dots,\mathbf{x}^{k-
		\delta}\in {\cal X}}{\min}\;\tilde F (\mathbf{x}^k,\dots,\mathbf{x}^{k-\delta})$. We are now ready to state our major convergence result.
\begin{theorem}\label{Theorem_convergence}  Let Problem \eqref{eq:optimization_problem} be given, along with  Algorithm 1 and the stochastic process $\underline{\boldsymbol{\omega}}$. Let $\{\underline{\mathbf{x}}^k\}_{k\in\mathbb{N}_+}$  be the
	sequence generated by the algorithm, given
	$\bx^0\in {\cal X}$.  Suppose that Assumptions A-C hold true and that
	\begin{equation}\label{eq:max gamma}
	 \gamma<\frac{c_{\tilde{f}}}{L_f+\frac{\delta^2L_f}{2}}.
	 \end{equation}
	  Define $K_\epsilon$ to be the first iteration such that $\mathbb{E}\left(\|M_F(\underline{\mathbf{x}}^k)\|_2^2\right)\leq\epsilon$. Then:

\noindent
(a)  Every limit point of $\{\underline{\bx}^k\}_{k\in\mathbb{N}_+}$ is a stationary solution  of  \eqref{eq:optimization_problem}
	a.s.;
	
	\noindent
	(b) The sequence of objective function values $\{ F(\underline{\mathbf{x}}^k)\}_{k\in\mathbb{N}_+}$ converges a.s.;\vspace{-0.3cm}

\begin{align}\label{eq:compKbound} (c)\;\;\; &
K_\epsilon\, \leq \, \frac{C_1(\gamma, \delta)(T+1)(F(\mathbf{x}^0)-F^*)}{\epsilon}\\&+\frac{C_2(\gamma, \delta)\gamma^2}{\epsilon}
\underbrace{	
	\sum\limits_{k=0}^{K_\epsilon}\mathbb{E}\left(\sum\limits_{i=1}^N\underline{M}_i^k\sum\limits_{t\in\underline{\mathcal{K}}_i^k}\left(\tilde{F}(\underline{\mathbf{x}}^t,\ldots,\underline{\mathbf{x}}^{t-\delta})-\tilde{F}(\underline{\mathbf{x}}^{t+1},\ldots,\underline{\mathbf{x}}^{t+1-\delta})\right)\right)}_\text{B},\nonumber
	\end{align}
	where: \vspace{-0.2cm}
	\begin{align}\label{eq:C1}
	&C_1(\gamma, \delta)\triangleq\frac{2\left(1+(1+L_E)(1+L_B+L_E)+\gamma^2Np_{\text{min}}\alpha^{-1}
	(1+(L_f+1)^2)\right)}{\gamma\left(c_{\tilde{f}}-\gamma\left(L_f+\frac{\delta^2L_f}{2}\right)\right)
	\left(p_{\text{min}}-p_{\text{min}}\alpha\right)},\\[0.2cm]
	&C_2(\gamma, \delta)\triangleq\frac{2TL_B(1+L_B+L_E)}{\gamma\left(c_{\tilde{f}}-\gamma\left(L_f+\frac{\delta^2L_f}{2}
	\right)\right)\left(p_{\text{min}}-p_{\text{min}}\alpha\right)},
	\label{eq:C2}	\end{align}
	$\alpha$ is an arbitrary
	fixed value in $(0;1)$, $\underline{M}_i^k\triangleq\underset{l=k,\ldots,k+T}{\max}|\bar{\underline{\mathcal{K}}}_i^l|$ .\vspace{-0.2cm}
\end{theorem}
\begin{proof}
	See Appendix.\vspace{-0.2cm}
\end{proof}	

\noindent
The theorem  states that convergence to stationary points occurs a.s. (the objective function values  converge too);  it also gives an estimate of the number of
 iterations $K_\epsilon$ necessary to enforce $\mathbb{E}\left(\|M_F(\mathbf{x}^k)\|_2^2\right)\le \epsilon$. Convergence is guaranteed if, in particular, the step-size is sufficiently small; the bound  \eqref{eq:max gamma} makes this precise. Note that if the method is synchronous,   $\delta =0$,
the bound in \eqref{eq:max gamma}, going like the inverse of the Lipschitz constant,  becomes the renowned conditions  used in many synchronous (proximal-gradient-like) schemes.   The term $\delta^2/2$ in the denominator of  \eqref{eq:max gamma} should then  be seen as the price to pay for asynchrony: the larger the possible delay $\delta$,  the smaller $\gamma$ should be to tolerate such delays. Roughly speaking, this means that the more chaotic the computational environment, the more conservative the step should be, and consequently the smaller the steps of the algorithm are.

The interpretation of the bound $\eqref{eq:compKbound}$ is not immediate,   because of the presence of the term $B$; we now elaborate on it.
If there exists a (deterministic) bound $C$ on $\underline{M}_i^k$, 
i.e.,   $\underline{M}_i^k \leq C$ for all $k$ and $i$, then  one can write
\begin{align*}
B \, &\leq \, C
\mathbb{E}\left(\sum\limits_{k=0}^{K_\epsilon}\sum\limits_{i=1}^N\sum\limits_{t\in\underline{\mathcal{K}}_i^k}\left(\tilde{F}(\underline{\mathbf{x}}^t,\ldots,\underline{\mathbf{x}}^{t-\delta})-\tilde{F}(\underline{\mathbf{x}}^{t+1},\ldots,\underline{\mathbf{x}}^{t+1-\delta})\right)\right)\\
&\leq\, C (T+\delta)(F(\bx^0)- F^*).
\end{align*}
Therefore, \eqref{eq:compKbound} can be upper bounded as\vspace{-0.2cm}
	\begin{equation}
	K_\epsilon\, \leq\, \Big[ C_1(\gamma, \delta)\cdot(T+1)+C_2(\gamma, \delta)\cdot\gamma^2\cdot C\cdot(T+\delta) \Big]\, \frac{F(\mathbf{x}^0)-F^*}{\epsilon}.
	\label{eq:compKbound_worst}
	\end{equation}
	Recalling the definition of $\underline{M}_i^k$ and  that $|\bar{\underline{\mathcal{K}}}_i^k|$
	is a random variable counting the number of times the index $i$ has been updated in the
iteration window $[k-\delta, k-1]$,   $\underline{M}_i^k \leq \delta$ always holds; therefore,  one can always take $C=\delta$.
 Of course this is a very rough approximation: it is hard to expect that in a
given time window  always the same variable,   $\bar i$,  is updated  and, even if this were the
case, all other $\underline{M}^k_i$, $i\neq \bar i$, would be $0$ and not $\delta$. Consider for example the commonly analyzed ``uniform case'' where  the processing  of every block-variable requires the same time. 
In this case one can reasonably take $C=1$ in \eqref{eq:compKbound_worst} {\em independently of the number of workers.} 

This   intuition is corroborated by our experiments, which are  summarized in Table \ref{tabled}.   AsyFLEXA was ran on two different architectures, namely: a shared-memory system  with 10 cores, and a message passing architecture composed of two nodes, with 10 cores each. Two   LASSO problems with $10,000$ variables each were considered, and the variables were equally partitioned across the workers.  In the first LASSO instance, the Hessian matrix was a dense matrix, which models situations where the workload is  equally distributed  across the workers. In the second LASSO problem,  the Hessian matrix had  many sparse rows, to create some unbalancedness in the workers' workload.  Table~\ref{tabled} shows the empirical  average   delay (the average is taken over the  components of the delay vector and time) and the maximum delay $\delta$, estimated in $500$ epochs (one epoch is triggered when all blocks have been updated once). As expected,
 $\delta$ is much larger  than the experienced average delay, confirming that  \eqref{eq:compKbound_worst} with $C=\delta$ is a very conservative bound.
\begin{table}
\begin{center}
\begin{tabular}{|*{4}{c|}}
	\hline
	&Average Delay&$\delta$&\# of cores\\
	\hline
	Multi-core Machine: Balanced Workload &1.11&3&10\\
	\hline
	Multi-core Machine: Unbalanced Workload &2.58&28&10\\
	\hline
	Message Passing System: Balanced Workload&1.87&30&10 per node\\
	\hline
	Message Passing System: Unbalanced Workload&3.01&36&10 per node\\
	\hline
\end{tabular}
\caption{{Average delay and maximum delay $\delta$ for AsyFLEXA,  ran on a multi-core machine and on a message passing system. 
}}\label{tabled}
\end{center}\vspace{-0.8cm}
\end{table}
While  $C$ can always  be pessimistically upper bounded by $\delta$, a tighter value can be found by tailoring the analysis to the specific  problem and architecture under consideration. 

Finally, we remark the importance of the use, in the complexity analysis, of the $\underline{M}^i_k$,   counting the number of times the index $i$ has been updated in a certain iteration window. The use of these variables  seems to  be a new feature of our analysis.
 While getting a sharp estimate for the upper bound $C$ may be difficult in practice, the bound
\eqref{eq:compKbound} gives a good insight into the elements that really influence  the algorithm, showing that what really matters, in some expressions appearing in \eqref{eq:compKbound}, is not  $\delta$, but the usually much smaller number of times the blocks  are actually updated. The use of these variables allows us to get a sharper bound with respect to the case in which one sets $C=\delta$. 
 From this point of view, we believe that typical upper bounds, as those obtained in
\cite{liu2015asynchronous,liu2015asyspcd,davis2016asynchronous,DavisEdmundsUdell,peng2016arock}, where $\delta$ is the only considered ``delay", do not give an accurate description of the actual  
worst-case scenario.}

\noindent \textit{Almost linear speedup:} To study the speedup achievable by the proposed method, 
we make two   simplifying assumptions, consistent with those made in the literature, namely: (a) {$\delta$   is proportional to the number of workers}, which is  reasonable in ``symmetric'' situations; and (b)
 $K_{\epsilon}$ is a good proxy for the number of iterations   performed by
the algorithm to reach the desired accuracy.
Choose the stepsize $\gamma$ to be small enough so that
\eqref{eq:max gamma} is always satisfied in the range of values of $\delta$ under consideration; then $C_1(\gamma, \delta)$ and $C_2(\gamma, \delta)$ can be taken to be
constants. Consider now the two summands in square brackets in
\eqref{eq:compKbound_worst}. Without the second term, one would have 
ideal linear speed-up. 
However, since one can expect the second term to be much smaller   than the
first (at least
when $\delta$ is not large), an almost linear speed-up can be anticipated. In fact, by \eqref{eq:C1} and  \eqref{eq:C2}, the second term is smaller than the first one, if $\gamma$ is sufficiently small.
In practice, of course, the speed up and in particular the range of the number of workers for which
linear speedup is expected, will be problem and architecture dependent.\vspace{-0.5cm}

		\section{Nonconvex Constraints} \label{nonconvex}\vspace{-0.2cm}
      In this section, we remove the assumption that
all constraints are convex, and study the following more general nonconvex constrained  optimization  problem:\vspace{-0.1cm}
		\begin{equation}
		\begin{aligned}
		& \underset{\mathbf{x}}{\min} & & F(\mathbf{x})= f(\mathbf{x})
+ \sum_{i=1}^N g_i(\bx_i)\\[0.3em]
&&&\begin{rcases}\mathbf{x}_i \in \mathcal X_i, \qquad i = 1, \dots , N,		\\[0.3em]
		c_{1}(\mathbf{x}_1)\leq0,\ldots, c_{N}(\mathbf{x}_N)\leq0, 		\end{rcases}\triangleq\mathcal{K}
		\end{aligned}
		\label{ncc_1}\tag{P$^\prime$}\vspace{-0.1cm}
		\end{equation}
		where   $c_{i}(\mathbf{x}_i)
		\leq0$ are   nonconvex private constraints, with $c_i: \mathcal{O}_i\rightarrow \mathbb{R}^{m_i}$, and $\mathcal{O}_i$ denoting  an open set containing $\mathcal{X}_i$; let also define $\mathcal{K}_i\triangleq\{\mathbf{x}_i
		\in\mathcal{X}_i:c_{i}(\mathbf{x}_i)\leq 0,\}$. 
Note that 	$c_{i}(\mathbf{x}_i)$ is a vector function, whose individual component   is denoted by	$c_{i,j}$, with $j =1, \ldots, m_i$.
		Problem  (\ref{ncc_1}) is motivated by several applications in
		signal processing, machine learning, and networking; see,  \cite{scutari2016parallel}
		and references therein for some concrete examples.
		
		
 To deal with nonconvex constraints, we   need some
\textit{regularity} of the constraint functions. Anticipating that all $c_i$ are assumed to be $C^1$ on $\mathcal X_i$, we will use the Mangasarian-Fromovitz Constraint
		Qualification (MFCQ). \vspace{-0.1cm}
	\begin{definition}	A point $\bar{\mathbf{x}}\in\mathcal{K}$ satisfies the   MFCQ if the following implication is
		satisfied:\vspace{-0.1cm}
		\begin{equation}
			\begin{rcases}
\mathbf{0}\in\sum\limits_{i=1}^N\sum\limits_{j\in\bar{J}_i}\mu_{i,j}\nabla_{\mathbf{x}}c_{i,j}(\bar{\mathbf{x}}_i)+N_\mathcal{X}
			(\bar{\mathbf{x}})\\
			\mu_{i,j}\geq 0,\,\forall j\in\bar{J}_i,\,\forall i\in\mathcal{N}
			\end{rcases}
			\Rightarrow\mu_{i,j}=0,\,\forall j\in\bar{J}_i,\,\forall i\in\mathcal{N},
		\end{equation}
		where $N_\mathcal{X}(\bar{\mathbf{x}})\triangleq\{\mathbf{z}\in\mathcal{X}:\mathbf{z}^\text{T}(\mathbf{y}-\bar{\mathbf{x}})\leq0,\,\forall\mathbf{y}\in\mathcal{X}\}$ is the normal cone to $\mathcal{X}$ at $\bar{\mathbf{x}}$, and $\bar{J}_i\triangleq\{j:c_{i,j}(\bar{\mathbf{x}}_i)=0\}$ is the index set of  nonconvex constraints that are active at $\bar{\mathbf{x}}_i$. \end{definition}
			  We study Problem~\eqref{ncc_1} under the following assumptions. \smallskip  	
	
		\noindent \textbf{Assumption A$^\prime$ (On the problem model).} Suppose that
		\begin{description}[topsep=-2.0pt,itemsep=-2.0pt]
\item[(A1$^\prime$)]  Each set $\mathcal{X}_i\subseteq\mathbb{R}^{n_i}$ is nonempty, closed, and convex;\smallskip
	\item[(A2$^\prime$)] $f:\mathcal{O}\rightarrow\mathbb{R}$ is $C^1$, where $\mathcal{O}$ is an open set containing $\mathcal{K}$;\smallskip
	\item[(A3$^\prime$)]  $\nabla_{\mathbf{x}_i} f$ is  $L_{f}$-Lipschitz continuous
	on $\mathcal{K}$;\smallskip
		\item[(A4$^\prime$)] Each $g_{i}:\mathcal{O}_i\rightarrow\mathbb{R}$ is convex, possibly nonsmooth, and $L_g$-Lipschitz continuous on $\mathcal{X}_i$, where $\mathcal{O}_i$ is an open set containing $\mathcal{X}_i$;\smallskip		
			\item[(A5$^\prime$)] $\mathcal{K}$ is a compact set;\smallskip
			\item[(A6$^\prime$)]  Each $c_{i,j}: \mathcal{O}_i\rightarrow \mathbb{R}$ is $C^1$; 
			\smallskip
			\item[(A7$^\prime$)] All feasible points of problem \eqref{ncc_1} satisfy the MFCQ.
		\end{description}
		
\smallskip

\noindent	 Assumptions A1$^\prime$-A4$^\prime$ are a duplication of A1-A4, repeated here for ease of reference;  A5$^\prime$ is   stronger  than A5, and made  here for the sake of
simplicity (one  could relax it with A5); and  A6$^\prime$ is a standard  differentiability assumption on the non convex constraints $c_{i,j}$.   \smallskip 		
			
\noindent \textbf{AsyFLEXA-NCC:} We are now ready to introduce our asynchronous algorithmic framework for (\ref{ncc_1}), termed AsyFLEXA-NCC (where NCC stands for Non Convex Constraints).	
 The method is still given by  Algorithm~\ref{alg:global},   with the only difference   that now  
  also the nonconvex constraints are replaced by suitably chosen convex approximations;  the probabilistic model concerning the choice of the the pair index-delays is the same as the one we used in the case of convex constraints, see Section
		\ref{sec:probabilistic}.
  More specifically, AsyFLEXA-NCC is given by Algorithm~\ref{alg:global} wherein the subproblem (\ref{eq:best-response-convex}) in Step 2    is replaced by
\begin{equation}
		\hat{\textbf{x}}_{i^k}(\textbf{x}^{k-\mathbf{d}^k})\triangleq \underset{\textbf{x}_{i^k} \in \mathcal{K}_{i^k}(\mathbf{x}^k_{i^k})}{\arg\min}
		\left\{{\tilde{F}_{i^k}(\textbf{x}_i;\textbf{x}^{k-\mathbf{d}^k})}\triangleq  \tilde f_{i^k}(\bx_i; {\bx}^{k-\mathbf{d}^k}) + g_{i^k}(\bx_i)\right\},
		\label{ncc_2}
		\end{equation}
		where $\tilde f_{i^k}$ is defined as in (\ref{eq:best-response-convex}); 
		$\mathcal{K}_{i^k}(\mathbf{x}_{i^k}^k)$ is a convex approximation of $\mathcal{K}_{i^k}$ at $\mathbf{x}^{k-\mathbf{d}^k}$, defined as  $$\mathcal{K}_{i^k}(\mathbf{x}^k_{i^k})\triangleq\{\mathbf{x}_i\in\mathcal{X}_{i^k}:\tilde{c}_{i^k,j}(\mathbf{x}_i;\mathbf{x}^k_{i^k})\leq 0,\, j=1, \ldots, m_{i^k}\};$$
		and  $\tilde{c}_{i^k,j}:\mathcal{X}_{i^k}\times\mathcal{K}_{i^k}\rightarrow\mathbb{R}$ is a suitably chosen surrogate of ${c}_{i^k,j}$. Note that  $K_{i^k}$ depends on $\mathbf{x}_{i^k}^{k}$ and not on $\mathbf{x}_{i^k}^{k-d_{i^k}^k}$,  because of Assumption C3 ($d_{i^k}^k=0$).


		The surrogate functions $\tilde{c}_{i^k,j}$ can be  chosen according to the following assumptions
		($\nabla \tilde{c}_{i,j}$ below denotes the partial gradient of $\tilde{c}_{i,j}$ with respect to the first argument).
		\smallskip

				\noindent \textbf{Assumption D (On the surrogate functions  $\tilde{c}_{i,j}$'s).}
				\begin{description}[topsep=-2.0pt,itemsep=-2.0pt]
					\item[  (D1)] Each $\tilde{c}_{i,j}(\bullet;\mathbf{y})$ $C^1$ on an open set containing $\mathcal{X}_i$, and convex on $\mathcal{X}_i$ for all
					$\mathbf{y}\in\mathcal{K}_i$;\smallskip
					\item[  (D2)] $\tilde{c}_{i,j}(\mathbf{y};\mathbf{y})=c_{i}(\mathbf{y})$, for all $\mathbf{y}\in\mathcal{K}_i$;
					\smallskip
					\item[  (D3)] $c_{i,j}(\mathbf{z})\leq\tilde{c}_{i,j}(\mathbf{z};\mathbf{y})$ for all
					$\mathbf{z}\in\mathcal{X}_i$ and $\mathbf{y}\in\mathcal{K}_i$;\smallskip
					\item[  (D4)] $\tilde{c}_{i,j}(\bullet;\bullet)$ is continuous on $\mathcal{X}_i\times\mathcal{K}_i$;\smallskip
					\item[  (D5)] $\nabla_{\mathbf{y}_i} c_{i,j}(\mathbf{y})=\nabla\tilde{c}_{i,j}(\mathbf{y};\mathbf{y})$, for all
					$\mathbf{y}\in\mathcal{K}_i$;\smallskip
					\item[  (D6)] $\nabla\tilde{c}_{i,j}(\bullet;\bullet)$ is continuous on $\mathcal{X}_i\times\mathcal{K}_i$;
					\smallskip
					\item[ (D7)] Each $\tilde{c}_{i,j}(\bullet;\bullet)$ is Lipschitz continuous on
					$\mathcal{X}_i\times\mathcal{K}_i$.\medskip
				\end{description}
Roughly speaking, Assumption D requires   $\tilde c_{i,j}$ to be an upper convex approximation of $c_{i,j}$ having the same gradient of $c_{i,j}$ at the base point $\bf y$. Finding such   approximations  is less difficult than it might seem at a first sight. Two examples are given below, while we  refer the reader to  \cite{facchinei2016feasible,scutari2016parallel} for a richer list. 

\smallskip

\noindent
$\bullet$
Suppose  $c_{i,j}$ has
a $L_{\nabla c_{i,j}}$-Lipschitz continuous gradient on the (compact) set  $\mathcal{K}_i$. By the Descent Lemma  \cite[Proposition A32]{Bertsekas_Book-Parallel-Comp},  the following convex approximation satisfies Assumption D:\vspace{-0.2cm}
\[
\tilde{c}_{i,j}(\bx;\by)\triangleq c_{i,j}(\by)+\nabla_{\bx}c_{i,j}(\by)\trt(\bx-\by)+
\frac{L_{\nabla c_{i,j}}}{2}\|\bx-\by\|^{2}_2\ge c_{i,j}(\bx).
\]

\noindent
$\bullet$
Suppose that  $c_{i,j}$ has a DC structure, that is,
$
c_{i,j}(\bx)=c_{i,j}^{+}(\bx)-c_{i,j}^{-}(\bx),
$
 $c_{i,j}^{+}$ and $c_{i,j}^{-}$ are two convex and continuously differentiable functions.
 By linearizing the concave part $-c_{i,j}^{-}$
and keeping the convex part $c_{i,j}^{+}$ unchanged, we obtain the
following convex upper approximation of $c_{i,j}$ that satisfies Assumption D:
\[
\tilde{c}_{i,j}(\bx;\by)\triangleq c_{i,j}^{+}(\bx)-c_{i,j}^{-}(\by)-\nabla_{\bx}c_{i,j}^{-}(\by)
\trt(\bx-\by)\ge c_{i,j}(\bx).\
\]
Note   that {the former example is quite general and in principle can be applied to
practically all constraints, even if it could be numerically undesirable if $L_{\nabla c_{i,j}}$ is too large;
the latter example covers, in a possibly more suitable way,  the case of concave constraints.}
\smallskip

		\noindent \textbf{AsyFLEXA-NCC: Convergence.}   In order to gauge convergence,   we
    redefine the stationarity measure $M_F$, to account for the presence of
nonconvex constraints. We use $\|M_F^c(\mathbf{x})\|_2$, with \vspace{-0.2cm}
\[
M_F^c(\mathbf{x})=\mathbf{x}-\underset{\mathbf{y}\in\mathcal{K}_1(\mathbf{x}_1)\times
\ldots\times\mathcal{K}_N(\mathbf{x}_N)}{\text{arg min}}\{\nabla f(\mathbf{x})^\text{T}
(\mathbf{y}-\mathbf{x})+g(\mathbf{y})+\frac{1}{2}\|\mathbf{y}-\mathbf{x}\|^2_2\}.
\]	
 It is a valid merit function: $\|M_F^c(\mathbf{x})\|_2$ is   continuous  and is zero only at
stationary solutions of \eqref{ncc_1} \cite{scutari2014distributed}.

		\begin{theorem}\label{th:ncc}
			Let Problem \eqref{ncc_1} be given, along with  AsyFLEXA-NCC and the stochastic process $\underline{\boldsymbol{\omega}}$. Let $\{\underline{\mathbf{x}}^k\}_{k\in\mathbb{N}_+}$  be the
			sequence generated by the algorithm, given
			$\bx^0\in {\cal K}$.  Suppose that Assumptions A',B-D hold  and that
			$
			\gamma
			$ is chosen as in (\ref{eq:max gamma}).
			Define $K_\epsilon$ to be the first iteration such that $\mathbb{E}\left(\|M_F^c(\underline{\mathbf{x}}^k)\|_2^2\right)\leq\epsilon$. Then: i) $\underline{\mathbf{x}}^k\in\mathcal{K}_1(\underline{\mathbf{x}}_1^k)\times\ldots\times\mathcal{K}_N(\underline{\mathbf{x}}^N)\subseteq\mathcal{K}$ for all $k\geq0$ (iterate feasibility); and ii) all results in
			Theorem \ref{Theorem_convergence} hold 
			with $M_F$ replaced by $M_F^c$. 
		\end{theorem} \vspace{-0.4cm}
			\begin{proof}  See Appendix~\ref{proof_theorem_2}.  \end{proof}
We are aware of only one other  BCD-asynchronous method  \cite{davis2016asynchronous,DavisEdmundsUdell}  able to deal with nonconvex
constraints. This  method  requires the ability to
find global minima of {\em nonconvex} subproblems while our scheme does not suffer from this drawbacks, {as} it only calls for the
solution of strongly convex subproblems. On the other hand, it needs a feasible starting
point and the ability to build approximations $\tilde c_{i,j}$ satisfying Assumption D. While our
requirements are   easier to be met in practice (and our analysis is based on a grounded probabilistic model), we think that  the two approaches complement
each other   and may cover different applications.\vspace{-0.4cm}

		\section{Conclusions}\vspace{-0.2cm}
		We proposed a novel  model  for the parallel block-descent asynchronous minimization of the sum of a
		nonconvex smooth function and a convex nonsmooth one, subject to nonconvex constraints.
		Our  model captures the essential features of modern multi-core architectures by providing
		a more realistic probabilistic description of asynchrony that that offered by the state of the art. Building on our new probabilistic  model, we proved   sublinear convergence rate of our
		algorithm and a near linear speedup when the  number of workers is not too large.
	 {While we performed some simple numerical tests to validate some of our theoretical findings, extensive simulations are beyond the scope of this paper, and will be the subject of a  subsequent work. Some preliminary numerical results can be found in \cite{cannelli2017asynchronous}.}
 		\vspace{-0.4cm}
	
\bibliographystyle{plain}
\addcontentsline{toc}{chapter}{Bibliography}
\bibliography{biblioFF}
 \vspace{-0.5cm}

\section{Appendix}	\vspace{-0.2cm}

\subsection{Preliminaries} \vspace{-0.2cm}\label{subsec:preliminary_results}
Hereafter,  we simplify the notation  using   $\tilde{\mathbf{x}}^k\triangleq \mathbf{x}^{k-\mathbf{d}^k}$.

\begin{asparaenum}
	\item[\bf 1. On conditional probabilities.]
	In our developments, we will consider the conditional expectation of random variables $\Z$ on $\Omega$ of the type
	$\E(\Z|{\cal F}^k)$. The following simple fact holds.
	\begin{proposition}
		Let $\mathbf{Z}$ be a random variable defined on $\Omega$, and  
		let ${\cal F}^k$ be defined in \eqref{eq:sub-sigma}. Then
		\begin{equation}\label{eq:expected value}
			\E\left(\Z|{\cal F}^k\right) = \sum_{(i,\d) \in {\cal N}\times {\cal D}} p\left((i,\mathbf{d}\right)\,|\boldsymbol{\omega}^{0:k})\Z\left((i,\mathbf{d}),\boldsymbol{\omega}^{0:k}\right).
		\end{equation}
	\end{proposition}
	\begin{proof}
		Recall that ${\cal F}^k$ is the $\sigma$-algebra generated by ${\cal C}^{k}$, which is a finite partition of $\Omega$. Therefore,  		 one can write \cite[Example 5.1.3]{durrett2010probability}
		$$
		\E\left(\Z|{\cal F}^k\right) = \frac{\E\left(\Z;C^k(\boldsymbol{\omega}^{0:k})\right)}{\mathbb{P}\left(C^k(\boldsymbol{\omega}^{0:k})\right)}.
		$$
The thesis   follows readily from \eqref{eq:conditional prob} and the fact that $\mathbf{Z}$   depends only on $\boldsymbol{\omega}^{0:k+1}$ and  takes a finite number of values.
	\end{proof}
	
	\item[\bf 2.  Properties of the best response $\hat{\mathbf{x}}(\cdot)$.]  We introduce next some basic properties of the best-response  maps defined in \eqref{eq:best-response-convex} and \eqref{ncc_2}.\smallskip
	
	\begin{proposition}[\cite{FLEXA}] \label{Prop_best_response} Given the best-response map $\hat{\mathbf{x}}(\cdot)\triangleq(\hat{\mathbf{x}}_i(\cdot))_{i=1}^N$, with  $\hat{\mathbf{x}}_i(\cdot)$   defined in \eqref{eq:best-response-convex}. Under Assumptions A-B, the following hold.
		\begin{description}
			
			\item [{(a) [Optimality]:}] For any $i \in \mathcal{N}$ and $\mathbf{y} \in \mathcal{X}$,
			\begin{equation}
				(\hat{\mathbf{x}}_i(\mathbf{y}) - \mathbf{y}_i)^T \nabla_{\mathbf{y}_i} f(\mathbf{y}) + g_i(\hat{\mathbf{x}}_i(\mathbf{y})) - g_i(\mathbf{y}_i) \leq - c_{\tilde{f}}\|\hat{\mathbf{x}}_i(\mathbf{y}) - \mathbf{y}_i\|_2^2 \,;
				\label{eq:lemma4}
			\end{equation}
			\item [{(b) [Lipschitz continuity]:}] For any $i \in \mathcal{N}$ and $\mathbf{y}, \mathbf{z}\in \mathcal{X}$,
			\begin{equation}
				\|\hat{\mathbf{x}}_i(\mathbf{y}) - \hat{\mathbf{x}}_i(\mathbf{z})\|_2 \leq L_{\hat{\mathbf{x}}}\|\mathbf{y} - \mathbf{z}\|_2,
			\end{equation}
			with $L_{\hat{\mathbf{x}}}={L_B}/{c_{\tilde{f}}}$;
			\item [{(c) [Fixed-point characterization]: }] The set of fixed-points of $\hat{\mathbf{x}}(\cdot)$ coincides with the set of stationary solutions of Problem~\eqref{eq:optimization_problem}. Therefore $\hat{\mathbf{x}}(\cdot)$ has at least one fixed point.
		\end{description}\smallskip
	\end{proposition}	
	\begin{proposition}[\cite{scutari2014distributed}]
		\label{Prop_best_response_ncc}
		Given the best-response map  $\hat{\mathbf{x}}(\cdot)\triangleq(\hat{\mathbf{x}}_i(\cdot))_{i=1}^N$, with   $\hat{\mathbf{x}}_i(\cdot)$   defined in \eqref{ncc_2}. Under Assumptions A$^\prime$-B-D,  the following hold.
		\begin{description}
			
			\item [{(a) [Optimality]:}] For any $i \in \mathcal{N}$ and $\mathbf{y} \in \mathcal{K}$,
			\begin{equation}
			(\hat{\mathbf{x}}_i(\mathbf{y}) - \mathbf{y}_i)^T \nabla_{\mathbf{y}_i} f(\mathbf{y}) + g_i(\hat{\mathbf{x}}_i(\mathbf{y})) - g_i(\mathbf{y}_i) \leq - c_{\tilde{f}}\|\hat{\mathbf{x}}_i(\mathbf{y}) - \mathbf{y}_i\|_2^2 \,;
			\end{equation}
			\item [{(b) [Lipschitz continuity]:}] For any $i \in \mathcal{N}$ and $\mathbf{y}, \mathbf{z}\in \mathcal{K}$,
			\begin{equation}
			\|\hat{\mathbf{x}}_i(\mathbf{y}) - \hat{\mathbf{x}}_i(\mathbf{z})\|_2 \leq \tilde{L}_{\hat{\mathbf{x}}}\|\mathbf{y} - \mathbf{z}\|_2^{1/2},
			\end{equation}
			with $\tilde{L}_{\hat{\mathbf{x}}}>0$.
		\end{description}\smallskip
	\end{proposition}
	
	\item[\bf 3. Young's Inequality \cite{young1912classes}.] For any   $\alpha, \mu_1, \mu_2 > 0$, there holds	\begin{equation}
		\mu_1\mu_2 \leq \frac{1}{2}(\alpha\mu_1^2 + \alpha^{-1}\mu_2^2).
		\label{eq:prel_3}
	\end{equation}

	\item[\bf 4. Representation of $\tilde{\mathbf{x}}^k$.]
Since at each iteration only one block of variables is updated,  $\tilde{\mathbf{x}}^k$ can be written as
	\begin{equation}
	\tilde{\mathbf{x}}^k_i=\mathbf{x}^k_i+\sum\limits_{l\in\bar{\mathcal{K}}_i^k}(\mathbf{x}^l_i-\mathbf{x}^{l+1}_i),
	\label{x_tilde}
	\end{equation}
	where $\bar{\mathcal{K}}^k_i$ is defined in Section~\ref{sec:convergence}.\end{asparaenum}

\subsection{Proof of Theorem \ref{Theorem_convergence}}\vspace{-0.3cm}
In this section, the best-response map $\hat{\mathbf{x}}(\cdot)$  is the one defined in \eqref{eq:best-response-convex}.

	\noindent For any given realization $\omega\in\Omega$ and $k\geq0$, the following holds:\vspace{-0.1cm}
	\begin{align}
	&\nonumber
	F(\mathbf{x}^{k+1}) \,
	=\, f(\mathbf{x}^{k+1}) + g(\mathbf{x}^{k+1})\\
&\nonumber	
	\stackrel{\text{(a)}}{=}\, f(\mathbf{x}^{k+1}) + \sum\limits_{i \neq i^k} g_i(\mathbf{x}_i^k) + g_{i^k}(\mathbf{x}_{i^k}^{k+1})\\
	&\nonumber
	\stackrel{\text{(b)}}{\leq} 
	\,f(\mathbf{x}^k) + \gamma\nabla_{\mathbf{x}_{i^k}} f (\tilde{\mathbf{x}}^k)^\text{T}(\hat{\mathbf{x}}_{i^k}(\tilde{\mathbf{x}}^k) - \mathbf{x}_{i^k}^k)+ \sum\limits_{i \neq i^k} g_i(\mathbf{x}_i^k)+ g_{i^k}(\mathbf{x}_{i^k}^{k+1})\\
	&\nonumber
	\quad \,+(\nabla_{\mathbf{x}_{i^k}} f (\mathbf{x}^k)-\nabla_{\mathbf{x}_{i^k}}f(\tilde{\mathbf{x}}^k))^\text{T}(\gamma(\hat{\mathbf{x}}_{i^k}(\tilde{\mathbf{x}}^k) - \mathbf{x}_{i^k}^k))+\frac{\gamma^2L_f}{2}\,\|\hat{\mathbf{x}}_{i^k}(\tilde{\mathbf{x}}^k) - \mathbf{x}_{i^k}^k\|_2^2\\
	&\nonumber
	\stackrel{\text{(c)}}{\leq}\, f(\mathbf{x}^k)  +\gamma\nabla_{\mathbf{x}_{i^k}} f (\tilde{\mathbf{x}}^k)^\text{T}(\hat{\mathbf{x}}_{i^k}(\tilde{\mathbf{x}}^k) - \tilde{\mathbf{x}}_{i^k}^k)\\
	&\nonumber
	\quad \,+(\nabla_{\mathbf{x}_{i^k}}f(\mathbf{x}^k)-\nabla_{\mathbf{x}_{i^k}}f(\tilde{\mathbf{x}}^k))^\text{T}(\gamma(\hat{\mathbf{x}}_{i^k}(\tilde{\mathbf{x}}^k) - \tilde{\mathbf{x}}_{i^k}^k))+ \frac{\gamma^2L_f}{2}\,\|\hat{\mathbf{x}}_{i^k}(\tilde{\mathbf{x}}^k) - \tilde{\mathbf{x}}_{i^k}^k\|_2^2\\
	&\nonumber
	\quad \,+ \sum\limits_{i \neq i^k} g_i(\mathbf{x}_i^k) + \gamma g_{i^k}(\hat{\mathbf{x}}_{i^k}(\tilde{\mathbf{x}}^k))+g_{i^k}(\mathbf{x}_{i^k}^k)-\gamma g_{i^k}(\tilde{\mathbf{x}}_{i^k}^k)
	\\
	&\nonumber
	\stackrel{\text{(d)}}{\leq} F(\mathbf{x}^k) -\gamma\,\left(c_{\tilde{f}}-\frac{\gamma L_f}{2}\right)\,\|\hat{\mathbf{x}}_{i^k}(\tilde{\mathbf{x}}^k)-\tilde{\mathbf{x}}_{i^k}^k\|_2^2\\
	&\nonumber\quad\,+L_f\,\|\mathbf{x}^k-\tilde{\mathbf{x}}^k\|_2\|\gamma(\hat{\mathbf{x}}_{i^k}(\tilde{\mathbf{x}}^k) - \tilde{\mathbf{x}}_{i^k}^k)\|_2\\
	&\nonumber
	\stackrel{\text{(e)}}{\leq} F(\mathbf{x}^k)-\gamma\,\left(c_{\tilde{f}}-\gamma L_f\right)\,\|\hat{\mathbf{x}}_{i^k}(\tilde{\mathbf{x}}^k)-\tilde{\mathbf{x}}_{i^k}^k\|_2^2+\frac{L_f}{2}\,\|\mathbf{x}^k-\tilde{\mathbf{x}}^k\|_2^2\\
	& \label{descent}
	\stackrel{\text{(f)}}{=} F(\mathbf{x}^k)-\gamma\,\left(c_{\tilde{f}}-\gamma L_f\right)\,\|\hat{\mathbf{x}}_{i^k}(\tilde{\mathbf{x}}^k)-\mathbf{x}_{i^k}^k\|_2^2+\frac{L_f}{2}\,\|\mathbf{x}^k-\tilde{\mathbf{x}}^k\|_2^2,
		\end{align}
	where  (a) follows from the updating rule of the algorithm; in (b) we used the Descent Lemma on $f$; (c) comes from the convexity of $g_i$ and C3; in (d) we used Proposition~\ref{Prop_best_response} and A3; (e) is due to the Young's inequality; and (f) is due to  C3.
	
	We bound $\|\mathbf{x}^k-\tilde{\mathbf{x}}^k\|_2^2$ as follows:\vspace{-0.2cm}
	\begin{align}
&\nonumber	\|\mathbf{x}^k-\tilde{\mathbf{x}}^k\|_2^2\,
	\stackrel{\text{(a)}}{\leq}\,\left(\sum\limits_{l=k-\delta}^{k-1}\|\mathbf{x}^{l+1}-\mathbf{x}^l\|_2\right)^2\,
	\stackrel{\text{(b)}}{\leq}\,\delta\sum\limits_{l=k-\delta}^{k-1}\|\mathbf{x}^{l+1}-\mathbf{x}^l\|_2^2\\
	& \label{eq:newbound1_multi}
	=\delta\left(\sum\limits_{l=k-\delta}^{k-1}\left(l-(k-1)+\delta\right)\,\|\mathbf{x}^{l+1}-\mathbf{x}^l\|_2^2 -\sum\limits_{l=k+1-\delta}^{k}\left(l-k+\delta\right)\,\|\mathbf{x}^{l+1}-\mathbf{x}^l\|_2^2\right)\\
	&\nonumber \quad+\delta^2\gamma^2\,\|\hat{\mathbf{x}}_{i^k}(\tilde{\mathbf{x}}^k)-\mathbf{x}_{i^k}^k\|_2^2,
	\end{align}
	where (a) comes from \eqref{x_tilde}; and (b) is due to the Jensen's inequality.
	
	Using \eqref{eq:newbound1_multi} in \eqref{descent}, the Lyapunov function \eqref{lyapunov}, and rearranging the terms, the following holds: for all $k\geq0$,
\begin{align}
	&\nonumber\tilde{F}(\mathbf{x}^{k+1}\ldots,\mathbf{x}^{k+1-\delta})
	\\&\leq \tilde{F}(\mathbf{x}^k,\ldots,\mathbf{x}^{k-\delta})-\gamma\,\left(c_{\tilde{f}}-\gamma\,\left(L_f+\frac{\delta^2L_f}{2}\right)\right)\,\|\hat{\mathbf{x}}_{i^k}(\tilde{\mathbf{x}}^k)-\mathbf{x}^k_{i^k}\|_2^2;
	\label{eq:final0_multi}\vspace{-0.2cm}
	\end{align}
and
\begin{align}
	&\nonumber\tilde{F}(\mathbf{x}^{k+T}\ldots,\mathbf{x}^{k+T-\delta})
	\leq \tilde{F}(\mathbf{x}^{k+T-1},\ldots,\mathbf{x}^{k+T-1-\delta})\\\nonumber&\quad-\gamma\,\left(c_{\tilde{f}}-\gamma\,\left(L_f+\frac{\delta^2L_f}{2}\right)\right)\,\|\hat{\mathbf{x}}_{i^{k+T-1}}(\tilde{\mathbf{x}}^{k+T-1})-\mathbf{x}^{k+T-1}_{i^{k+T-1}}\|_2^2\\
	&\leq\tilde{F}(\mathbf{x}^k,\ldots,\mathbf{x}^{k-\delta})-\gamma\left(c_{\tilde{f}}-\gamma\,\left(L_f+\frac{\delta^2L_f}{2}\right)\right)\,\sum\limits_{t=k}^{k+T-1}\|\hat{\mathbf{x}}_{i^t}(\tilde{\mathbf{x}}^t)-\mathbf{x}^t_{i^t}\|_2^2.
	\end{align}
Taking conditional expectation both sides we have that the following holds a.s.:
\begin{align}
&\nonumber	\mathbb{E}\left(\tilde{F}(\underline{\mathbf{x}}^{k+T}\ldots,\underline{\mathbf{x}}^{k+T-\delta})|\mathcal{F}^{k-1}\right)\leq\tilde{F}(\underline{\mathbf{x}}^k,\ldots,\underline{\mathbf{x}}^{k-\delta})\\&-\gamma\left(c_{\tilde{f}}-\gamma\,\left(L_f+\frac{\delta^2L_f}{2}\right)\right)\,\sum\limits_{t=k}^{k+T-1}\mathbb{E}\left(\|\hat{\mathbf{x}}_{\underline{i}^t}(\underline{\tilde{\mathbf{x}}}^t)-\mathbf{x}^t_{\underline{i}^t}\|_2^2|\mathcal{F}^{t-1}\right).
\label{eq:final0_multi2}
\end{align}
	Using \eqref{eq:max gamma},  \eqref{eq:final0_multi2}, A5, and the Martingale's theorem \cite{robbins1985convergence}, we deduce that i) $\{\tilde{F}(\underline{\mathbf{x}}^k,\ldots,\underline{\mathbf{x}}^{k-\delta})\}_{k\in\mathbb{N}_+}$, and thus $\{{F}(\underline{\mathbf{x}}^k,\ldots,\underline{\mathbf{x}}^{k-\delta})\}_{k\in\mathbb{N}_+}$ converge a.s., ii) $\{\underline{\mathbf{x}}^k\}_{k\in\mathbb{N}_+}$ is bounded on $\mathcal{X}$ a.s., and iii)
	\begin{equation}
		\lim\limits_{k\to+\infty}\sum\limits_{t=k}^{k+T-1}\mathbb{E}\left(\|\hat{\mathbf{x}}_{\underline{i}^t}(\underline{\tilde{\mathbf{x}}}^t)-\underline{\mathbf{x}}_{\underline{i}^t}^t\|_2|\mathcal{F}^{t-1}\right)=0,\quad\text{ a.s.}
		\label{limit}
	\end{equation}
	From \eqref{limit}, it follows that there exists a set $\bar{\Omega}\subseteq\Omega$, with $\mathbb{P}(\bar{\Omega})=1$, such that for any $\omega\in\bar{\Omega}$,  \vspace{-0.2cm}
	\begin{align}
		\nonumber&\sum\limits_{t=k}^{k+T-1}\mathbb{E}\left(\|\hat{\mathbf{x}}_{\underline{i}^t}(\underline{\tilde{\mathbf{x}}}^t)-\mathbf{x}^t_{\underline{i}^t}\|_2^2|\mathcal{F}^{t-1}\right)\\&\nonumber\stackrel{(a)}{=}\sum\limits_{t=k}^{k+T-1}\sum\limits_{(i,\mathbf{d})\in\mathcal{N}\times\mathcal{D}}p\left((i,\mathbf{d})|\boldsymbol{\omega}^{0:t-1}\right)\,\|\hat{\mathbf{x}}_i(\tilde{\mathbf{x}}^t)-\mathbf{x}^t_i\|_2\\
		&\stackrel{(b)}{\geq}p_{\text{min}}\sum\limits_{i=1}^N\|\hat{\mathbf{x}}_i(\tilde{\mathbf{x}}^{k+t_k(i)})-\mathbf{x}_i^{k+t_k(i)}\|_2,
		\label{limit2}
	\end{align}
	where in (a) we used \eqref{eq:expected value}; and in (b) we used C2 and defined $t_k(i)\triangleq\min\{t\in[0;T]|p(i|\omega^{0:t+k-1})\geq p_{\text{min}}\}$. We also have:\vspace{-0.2cm}
	\begin{align}
		\nonumber &\|\hat{\mathbf{x}}(\mathbf{x}^k)-\mathbf{x}^k\|_2 \leq\sum\limits_{i=1}^N\|\hat{\mathbf{x}}_i(\mathbf{x}^k)-\mathbf{x}_i^k\|_2\\
		\nonumber&\stackrel{(a)}{\leq}\sum\limits_{i=1}^N\left(\|\hat{\mathbf{x}}_i(\tilde{\mathbf{x}}^{k+t_k(i)})-\mathbf{x}_i^{k+t_k(i)}\|_2+(1+L_{\hat{\mathbf{x}}})\,\|\tilde{\mathbf{x}}^{k+t_k(i)}-\mathbf{x}^k\|_2\right)\\
		\nonumber&\stackrel{(b)}{\leq}\sum\limits_{i=1}^N\Biggl(\|\hat{\mathbf{x}}_i(\tilde{\mathbf{x}}^{k+t_k(i)})-\mathbf{x}_i^{k+t_k(i)}\|_2+(1+L_{\hat{\mathbf{x}}})\biggl(\|\mathbf{x}^{k+t_k(i)}-\mathbf{x}^k\|_2 \\&\nonumber\quad+\sum\limits_{l=k+t_k(i)-\delta}^{k+t_k(i)-1}\|\mathbf{x}^{l+1}-\mathbf{x}^l\|_2\biggr)\Biggr)\\
		&\stackrel{(c)}{\leq}\sum\limits_{i=1}^N\left(\|\hat{\mathbf{x}}_i(\tilde{\mathbf{x}}^{k+t_k(i)})-\mathbf{x}_i^{k+t_k(i)}\|_2+2\gamma(1+L_{\hat{\mathbf{x}}})\sum\limits_{l=k-\delta}^{k+T-1}\|\hat{\mathbf{x}}_{i^l}(\tilde{\mathbf{x}}^l)-\mathbf{x}_{i^l}^l\|_2\right),
		\label{final}
	\end{align}
	where in (a) we used Proposition \ref{Prop_best_response}; (b) comes from \eqref{x_tilde}; and (c) from the updating rule of the algorithm.
	We deduce from  \eqref{limit}, \eqref{limit2}, and \eqref{final},  that \vspace{-0.1cm}	\begin{equation}
		\lim\limits_{k\to+\infty}\|\hat{\mathbf{x}}(\mathbf{x}^k)-\mathbf{x}^k\|_2=0.
		\label{finallimit}\vspace{-0.2cm}
	\end{equation}
	Since the sequence $\{\mathbf{x}^k\}_{k\in\mathbb{N}_+}$ is bounded, it has at least one limit point $\bar{\mathbf{x}}$ that belongs to $\mathcal{X}$. By the continuity of $\hat{\mathbf{x}}(\cdot)$ (see Proposition \ref{Prop_best_response}) and \eqref{finallimit}, it must be $\hat{\mathbf{x}}(\bar{\mathbf{x}})=\bar{\mathbf{x}}$, and thus by  Proposition \ref{Prop_best_response}  $\bar{\mathbf{x}}$ is a stationary solution of Problem~\eqref{eq:optimization_problem}.  Since \eqref{finallimit} holds for any $\omega\in\bar{\Omega}$, the previous results hold a.s..
	
	Let us now define:\vspace{-0.1cm}
	\begin{equation}
	\hat{\mathbf{y}}_i(\mathbf{x}^k)=\underset{\mathbf{y}_i\in\mathcal{X}_i} {\text{argmin}}\left\{\nabla_{\mathbf{x}_i}f(\mathbf{x}^k)^\text{T}(\mathbf{y}_i-\mathbf{x}_i^k)+g_i(\mathbf{y}_i)+\frac{1}{2}\,\|\mathbf{y}_i-\mathbf{x}_i^k\|^2_2\right\},
	\label{34}
	\end{equation}
	and note that $M_F(\mathbf{x})=[\mathbf{x}_1^k-\hat{\mathbf{y}}_1(\mathbf{x}^k),\ldots,\mathbf{x}_N^k-\hat{\mathbf{y}}_N(\mathbf{x}^k)]^\text{T}$. It is easy to check that $\hat{\mathbf{y}}(\cdot)$ is $L_{\hat{\mathbf{y}}}$-Lipschitz continuous on $\mathcal{X}$, with $L_{\hat{\mathbf{y}}}\triangleq L_f+1$. Fix a realization $\omega\in\Omega$.   The  optimality  of $\hat{\mathbf{y}}_{i^k}(\mathbf{x}^k)$ along with the convexity of $g_{i^k}$, leads
	\begin{align}
	&\left(\nabla_{\mathbf{x}_{i^k}} f(\mathbf{x}^k) + \hat{\mathbf{y}}_{i^k}(\mathbf{x}^k)- \mathbf{x}^k_{i^k}\right)^\text{T}\left(\hat{\mathbf{x}}_{i^k}(\tilde{\mathbf{x}}^k) - \hat{\mathbf{y}}_{i^k}(\mathbf{x}^k)\right)\\&\quad + g_{i^k}(\hat{\mathbf{x}}_{i^k}(\tilde{\mathbf{x}}^k)) - g_{i^k}(\hat{\mathbf{y}}_{i^k}(\mathbf{x}^k)) \geq 0.
	\label{eq:t1}
	\end{align}
	Similarly,  one can write  for $\hat{\mathbf{x}}_{i^k}(\tilde{\mathbf{x}}^k)$: \vspace{-0.1cm}
	\begin{equation}
	\nabla \tilde{f}_{i^k}(\hat{\mathbf{x}}_{i^k}(\tilde{\mathbf{x}}^k);\tilde{\mathbf{x}}^k)^\text{T}\left(\hat{\mathbf{y}}_{i^k}(\mathbf{x}^k) - \hat{\mathbf{x}}_{i^k}(\tilde{\mathbf{x}}^k)\right) + g_{i^k}(\hat{\mathbf{y}}_{i^k}(\mathbf{x}^k)) - g_{i^k}(\hat{\mathbf{x}}_{i^k}(\tilde{\mathbf{x}}^k))) \geq 0.
	\label{eq:t2}
	\end{equation}
	 Summing \eqref{eq:t1} and \eqref{eq:t2}, adding and subtracting $\hat{\mathbf{x}}_{i^k}(\tilde{\mathbf{x}}^k)$,  and 
	  using the gradient consistency B2, yield \vspace{-0.2cm}
	\begin{align}
	&\nonumber\left(\nabla \tilde{f}_{i^k}(\mathbf{x}^k_{i^k};\mathbf{x}^k) - \nabla \tilde{f}_{i^k}(\hat{\mathbf{x}}_{i^k}(\tilde{\mathbf{x}}^k);\tilde{\mathbf{x}}^k) + \hat{\mathbf{x}}_{i^k}(\tilde{\mathbf{x}}^k) - \mathbf{x}^k_{i^k}\right)^\text{T}\\&\left(\hat{\mathbf{x}}_{i^k}(\tilde{\mathbf{x}}^k) - \hat{\mathbf{y}}_{i^k}(\mathbf{x}^k)\right)
	\geq \|\hat{\mathbf{x}}_{i^k}(\tilde{\mathbf{x}}^k) - \hat{\mathbf{y}}_{i^k}(\mathbf{x}^k)\|_2^2.
	\label{eq:t4}
	\end{align}
	Summing and subtracting $\nabla \tilde{f}_{i^k}(\hat{\mathbf{x}}_{i^k}(\tilde{\mathbf{x}}^k);\mathbf{x}^k)$ and using  the triangular inequality, the LHS of \eqref{eq:t4} can be upper bounded as
	\begin{align}	
	&	\nonumber \|\nabla \tilde{f}_{i^k}(\hat{\mathbf{x}}_{i^k}(\tilde{\mathbf{x}}^k);\mathbf{x}^k) - \nabla \tilde{f}_{i^k}(\hat{\mathbf{x}}_{i^k}(\tilde{\mathbf{x}}^k);\tilde{\mathbf{x}}^k)\|_2\\ & \label{eq:t6}\quad+ \|\nabla \tilde{f}_{i^k}(\mathbf{x}^k_{i^k};\mathbf{x}^k) - \nabla \tilde{f}_{i^k}(\hat{\mathbf{x}}_{i^k}(\tilde{\mathbf{x}}^k);\mathbf{x}^k)\|_2		
	+ \|\hat{\mathbf{x}}_{i^k}(\tilde{\mathbf{x}}^k) - \mathbf{x}^k_{i^k}\|_2
	\\&\nonumber\geq \|\hat{\mathbf{x}}_{i^k}(\tilde{\mathbf{x}}^k) -
	\hat{\mathbf{y}}_{i^k}(\mathbf{x}^k)\|_2.	\end{align}
	We can further upper-bound the left hand side invoking    B3 and B4, and write:
	\begin{equation}
	\|\hat{\mathbf{x}}_{i^k}(\tilde{\mathbf{x}}^k) - \hat{\mathbf{y}}_{i^k}(\mathbf{x}^k)\|_2 \leq (1+L_E)\,\|\hat{\mathbf{x}}_{i^k}(\tilde{\mathbf{x}}^k) - \mathbf{x}^k_{i^k}\|_2 +L_B\,\| \mathbf{x}^k-\tilde{\mathbf{x}}^k\|_2\, .
	\label{eq:t7}
	\end{equation}
	Finally, squaring  both sides, we get
	\begin{align}
	&\nonumber
	\|\hat{\mathbf{x}}_{i^k}(\tilde{\mathbf{x}}^k) - \hat{\mathbf{y}}_{i^k}(\mathbf{x}^k)\|_2^2\;		
	\leq\; (1+L_E)^2\,\|\hat{\mathbf{x}}_{i^k}(\tilde{\mathbf{x}}^k) - \mathbf{x}^k_{i^k}\|_2^2 + L_B^2\,\| \mathbf{x}^k-\tilde{\mathbf{x}}^k \|_2^2\\
	&+ 2L_B(1+L_E)\,\|\hat{\mathbf{x}}_{i^k}(\tilde{\mathbf{x}}^k) - \mathbf{x}^k_{i^k}\|_2\,\| \mathbf{x}^k-\tilde{\mathbf{x}}^k\|_2.
	\label{eq:t8}
	\end{align}
	We bound next the term  $\|\mathbf{x}_{i^k}^k - \hat{\mathbf{y}}_{i^k}(\mathbf{x}^k)\|_2^2$. We write
	\begin{align}
	&\nonumber\|\mathbf{x}_{i^k}^k - \hat{\mathbf{y}}_{i^k}(\mathbf{x}^k)\|_2^2
	= \|\mathbf{x}_{i^k}^k - \hat{\mathbf{x}}_{i^k}(\tilde{\mathbf{x}}^k) +\hat{\mathbf{x}}_{i^k}(\tilde{\mathbf{x}}^k)- \hat{\mathbf{y}}_{i^k}(\mathbf{x}^k)\|_2^2 \\
	&	\nonumber
	\leq \;2\,\left(\|\hat{\mathbf{x}}_{i^k}(\tilde{\mathbf{x}}^k) - \mathbf{x}_{i^k}^k\|_2^2 + \|\hat{\mathbf{x}}_{i^k}(\tilde{\mathbf{x}}^k)- \hat{\mathbf{y}}_{i^k}(\mathbf{x}^k)\|_2^2\right)
	\\&\nonumber\stackrel{(a)}{\leq} \left(2+2(1+L_E)^2\right)\,\|\hat{\mathbf{x}}_{i^k}(\tilde{\mathbf{x}}^k) - \mathbf{x}_{i^k}^k\|_2^2+ 2L_B^2\,\| \mathbf{x}^k-\tilde{\mathbf{x}}^k \|_2^2\\
	&\nonumber	\quad 	 + 4L_B(1+L_E)\,\|\hat{\mathbf{x}}_{i^k}(\tilde{\mathbf{x}}^k) - \mathbf{x}^k_{i^k}\|_2\,\| \mathbf{x}^k-\tilde{\mathbf{x}}^k \|_2\\
	&\nonumber
	\stackrel{(b)}{\leq} 2\left(1+(1+L_E)(1+L_B+L_E)\right)\,\|\hat{\mathbf{x}}_{i^k}(\tilde{\mathbf{x}}^k)-\mathbf{x}^k_{i^k}\|_2^2\\
	&\label{eq:t9}\quad+2L_B(1+L_B+L_E)\,\|\mathbf{x}^k-\tilde{\mathbf{x}}^k\|_2^2,
	\end{align}
	where (a) comes from \eqref{eq:t8}; and (b) follows from the Young's inequality.
	Note that \vspace{-0.2cm}
	\begin{align}
	&\nonumber\|\mathbf{x}^k-\tilde{\mathbf{x}}^k\|_2^2\,
	=\,\sum\limits_{i=1}^N\|\mathbf{x}^k_i-\tilde{\mathbf{x}}^k_i\|_2^2\,
	\stackrel{(a)}{\leq}\,\sum\limits_{i=1}^N\left(\sum\limits_{l\in\bar{\mathcal{K}}_i^k}\|\mathbf{x}^{l+1}-\mathbf{x}^l\|_2\right)^2\\
	 & \stackrel{(b)}{\leq}\sum\limits_{i=1}^N\bar{M}_i^k\sum\limits_{l\in\bar{\mathcal{K}}_i^k}\|\mathbf{x}^{l+1}-\mathbf{x}^l\|_2^2\,
	=\,\gamma^2\sum\limits_{i=1}^N\bar{M}_i^k\sum\limits_{l\in\bar{\mathcal{K}}_i^k}\|\hat{\mathbf{x}}_{i^l}(\tilde{\mathbf{x}}^l)-\mathbf{x}_{i^l}^l\|_2^2,
	\label{bound}	
	\end{align}
	where (a)  comes from \eqref{x_tilde}; and in (b) we used the Jensen's inequality  {and defined $\bar{M}_i^k\triangleq|\bar{\mathcal{K}}_i^k|$}.
	Combining \eqref{eq:t9} and \eqref{bound}, we get:
	\begin{align}	\nonumber
	\|\mathbf{x}_{i^k}^k &- \hat{\mathbf{y}}_{i^k}(\mathbf{x}^k)\|_2^2\,
	\leq\, 2\left(1+(1+L_E)(1+L_B+L_E)\right)\,\|\hat{\mathbf{x}}_{i^k}(\tilde{\mathbf{x}}^k)-\mathbf{x}^k_{i^k}\|_2^2\\ &+2\gamma^2L_B(1+L_B+L_E)\sum\limits_{i=1}^N\bar{M}_i^k\sum\limits_{l\in\bar{\mathcal{K}}_i^k}\|\hat{\mathbf{x}}_{i^l}(\tilde{\mathbf{x}}^l)-\mathbf{x}_{i^l}^l\|_2^2.
	\label{bound2}
	\end{align}	
	We take now  the conditional expectation of the term on the LHS of \eqref{bound2}, and obtain
	\begin{align}
	&
	\nonumber
	\sum\limits_{t=k}^{k+T}\mathbb{E}\left(	\|\mathbf{x}_{\underline{i}^t}^t - \hat{\mathbf{y}}_{\underline{i}^t}(\mathbf{x}^t)\|_2^2|\mathcal{F}^{t-1}\right)(\omega)\,
	\stackrel{(a)}{=}\,\sum\limits_{t=k}^{k+T}\sum\limits_{i=1}^Np(i|\omega^{0:t-1})\,\|\mathbf{x}_i^t - \hat{\mathbf{y}}_i(\mathbf{x}^t)\|_2^2\\
	& \	\label{new_bound}\stackrel{(b)}{\geq}\sum\limits_{i=1}^Np_{\text{min}}\,\|\mathbf{x}_i^{k+t_k(i)}-\hat{\mathbf{y}}_i(\mathbf{x}^{k+t_k(i)})\|_2^2\\ &
	\nonumber
	\stackrel{(c)}{\geq} p_{\text{min}}\sum\limits_{i=1}^N\left(\|\mathbf{x}_i^k-\hat{\mathbf{y}}_i(\mathbf{x}^k)\|_2-\|\mathbf{x}_i^{k+t_k(i)}-\hat{\mathbf{y}}_i(\mathbf{x}^{k+t_k(i)})-\mathbf{x}_i^k+\hat{\mathbf{y}}_i(\mathbf{x}^k)\|_2\right)^2\\&
	\nonumber
	\geq p_{\text{min}}\sum\limits_{i=1}^N\Bigl(\|\mathbf{x}_i^k-\hat{\mathbf{y}}_i(\mathbf{x}^k)\|_2^2
	\\&\nonumber\quad-2\|\mathbf{x}_i^k-\hat{\mathbf{y}}_i(\mathbf{x}^k)\|_2\,\|\mathbf{x}_i^{k+t_k(i)} -\hat{\mathbf{y}}_i(\mathbf{x}^{k+t_k(i)})-\mathbf{x}_i^k+\hat{\mathbf{y}}_i(\mathbf{x}^k)\|_2\Bigr),
	\end{align}
	where in (a) we used \eqref{eq:expected value};   (b) follows from  C2; and in (c) we used the reverse triangle inequality. By \eqref{new_bound} and \eqref{bound2}, we obtain:\vspace{-0.2cm}
	\begin{align}\nonumber &
	p_{\text{min}}\sum\limits_{i=1}^N\|\mathbf{x}_i^k-\hat{\mathbf{y}}_i(\mathbf{x}^k)\|_2^2\,
	=\,	p_{\text{min}}\,\|M_F(\mathbf{x}^k)\|^2_2\\\nonumber &
	\leq\sum\limits_{t=k}^{k+T}\Biggr(2\left(1+(1+L_E)(1+L_B+L_E)\right)\,\mathbb{E}\left(\|\hat{\mathbf{x}}_{\underline{i}^t}(\underline{\tilde{\mathbf{x}}}^t)-\mathbf{x}^t_{\underline{i}^t}\|_2^2|\mathcal{F}^{t-1}\right)(\omega)\\
	\nonumber &
	\quad+2\gamma^2L_B(1+L_B+L_E)\sum\limits_{i=1}^N\bar{M}_i^t\sum\limits_{l\in\bar{\mathcal{K}}_i^t}\|\hat{\mathbf{x}}_{i^l}(\tilde{\mathbf{x}}^l)-\mathbf{x}_{i^l}^l\|_2^2\Biggl)\\
	\nonumber &
	\quad+2p_{\text{min}}\sum\limits_{i=1}^N\|\mathbf{x}_i^k-\hat{\mathbf{y}}_i(\mathbf{x}^k)\|_2\,\|\mathbf{x}_i^{k+t_k(i)}-\hat{\mathbf{y}}_i(\mathbf{x}^{k+t_k(i)})-\mathbf{x}_i^k+\hat{\mathbf{y}}_i(\mathbf{x}^k)\|_2\\
	\nonumber &
	\stackrel{(a)}{\leq}2\left(1+(1+L_E)(1+L_B+L_E)\right)\,\sum\limits_{t=k}^{k+T}\mathbb{E}\left(\|\hat{\mathbf{x}}_{\underline{i}^t}(\underline{\tilde{\mathbf{x}}}^t)-\mathbf{x}^t_{\underline{i}^t}\|_2^2|\mathcal{F}^{t-1}\right)(\omega)\\
	\nonumber &
	\quad+2T\gamma^2L_B(1+L_B+L_E)\,\sum\limits_{i=1}^NM_i^k\sum\limits_{l\in\mathcal{K}_i^k}\|\hat{\mathbf{x}}_{i^l}(\tilde{\mathbf{x}}^l)-\mathbf{x}_{i^l}^l\|_2^2
	+p_{\text{min}}\alpha\,\|M_F(\mathbf{x}^k)\|_2^2\\
	\nonumber &\quad+p_{\text{min}}\alpha^{-1}\sum\limits_{i=1}^N\|\mathbf{x}_i^{k+t_k(i)}-\hat{\mathbf{y}}_i(\mathbf{x}^{k+t_k(i)})-\mathbf{x}_i^k+\hat{\mathbf{y}}_i(\mathbf{x}^k)\|_2^2\\
	\nonumber &
	\stackrel{(b)}{\leq}2\left(1+(1+L_E)(1+L_B+L_E)\right)\sum\limits_{t=k}^{k+T}\mathbb{E}\left(\|\hat{\mathbf{x}}_{\underline{i}^t}(\underline{\tilde{\mathbf{x}}}^t)-\mathbf{x}^t_{\underline{i}^t}\|_2^2|\mathcal{F}^{t-1}\right)(\omega)\\
	\nonumber &
	\quad+2T\gamma^2L_B(1+L_B+L_E)\sum\limits_{i=1}^NM_i^k\sum\limits_{l\in\mathcal{K}_i^k}\|\hat{\mathbf{x}}_{i^l}(\tilde{\mathbf{x}}^l)-\mathbf{x}_{i^l}^l\|_2^2
	+p_{\text{min}}\alpha\,\|M_F(\mathbf{x}^k)\|_2^2\\
	\nonumber &\quad
	+2p_{\text{min}}\alpha^{-1}\sum\limits_{i=1}^N\left(\|\mathbf{x}^{k+t_k(i)}_i-\mathbf{x}_i^k\|_2^2+\|\hat{\mathbf{y}}_i(\mathbf{x}^{k+t_k(i)})-\hat{\mathbf{y}}_i(\mathbf{x}^k)\|_2^2\right)\\
	\nonumber &
	\stackrel{(c)}{\leq}2\left(1+(1+L_E)(1+L_B+L_E)\right)\sum\limits_{t=k}^{k+T}\mathbb{E}\left(\|\hat{\mathbf{x}}_{\underline{i}^t}(\underline{\tilde{\mathbf{x}}}^t)-\mathbf{x}^t_{\underline{i}^t}\|_2^2|\mathcal{F}^{t-1}\right)(\omega)\\
	\nonumber &
	\quad+2T\gamma^2L_B(1+L_B+L_E)\sum\limits_{i=1}^NM_i^k\sum\limits_{l\in\mathcal{K}_i^k}\|\hat{\mathbf{x}}_{i^l}(\tilde{\mathbf{x}}^l)-\mathbf{x}_{i^l}^l\|_2^2
	+p_{\text{min}}\alpha\,\|M_F(\mathbf{x}^k)\|_2^2\\
	\nonumber &
	\quad+2\gamma^2p_{\text{min}}\alpha^{-1}(1+L_{\hat{\mathbf{y}}}^2)\sum\limits_{i=1}^N\sum\limits_{l=k}^{k+t_k(i)-1}\|\hat{\mathbf{x}}_{i^l}(\tilde{\mathbf{x}}^l)-\mathbf{x}^l_{i^l}\|_2^2
	\\
	\nonumber &
	\leq2\left(1+(1+L_E)(1+L_B+L_E)\right)\sum\limits_{t=k}^{k+T}\mathbb{E}\left(\|\hat{\mathbf{x}}_{\underline{i}^t}(\underline{\tilde{\mathbf{x}}}^t)-\mathbf{x}^t_{\underline{i}^t}\|_2^2|\mathcal{F}^{t-1}\right)(\omega)\\
	\nonumber &
	\quad+2T\gamma^2L_B(1+L_B+L_E)\sum\limits_{i=1}^NM_i^k\sum\limits_{l\in\mathcal{K}_i^k}\|\hat{\mathbf{x}}_{i^l}(\tilde{\mathbf{x}}^l)-\mathbf{x}_{i^l}^l\|_2^2
	+p_{\text{min}}\alpha\,\|M_F(\mathbf{x}^k)\|_2^2\\
	&\quad+2\gamma^2p_{\text{min}}\alpha^{-1}(1+L_{\hat{\mathbf{y}}}^2)N\sum\limits_{l=k}^{k+T-1}\|\hat{\mathbf{x}}_{i^l}(\tilde{\mathbf{x}}^l)-\mathbf{x}^l_{i^l}\|_2^2,
	\end{align}
	where in (a) we used the Young's inequality   {and the definition of $M_i^k$ (cf.  Section~\ref{sec:convergence})};  in (b) we used the triangle and  Jensen's inequalities; and (c) comes from the updating rule of the algorithm.
	Rearranging the terms and taking expectation of both sides,  we  get:
	\begin{align}
	&\nonumber
	\mathbb{E}\left(\|M_F(\underline{\mathbf{x}}^k)\|^2_2\right)\\
	&\nonumber
	\leq\frac{2\left(1+(1+L_E)(1+L_B+L_E)+\gamma^2Np_{\text{min}}\alpha^{-1}(1+L_{\hat{\mathbf{y}}}^2)\right)}{p_{\text{min}}-\alpha p_{\text{min}}}\\&\nonumber\sum\limits_{t=k}^{k+T}\mathbb{E}\left(\|\hat{\mathbf{x}}_{\underline{i}^t}(\underline{\tilde{\mathbf{x}}}^t)-\mathbf{x}^t_{\underline{i}^t}\|_2^2\right)\\
	&
	+\frac{2T\gamma^2 L_B(1+L_B+L_E)}{p_{\text{min}}-\alpha p_{\text{min}}}\,\mathbb{E}\left(\sum\limits_{i=1}^N\underline{M}_i^k\sum\limits_{t\in\underline{\mathcal{K}}_i^k}\|\hat{\mathbf{x}}_{\underline{i}^t}(\underline{\tilde{\mathbf{x}}}^t)-\mathbf{x}^t_{\underline{i}^t}\|_2^2\right).\label{final_iter}
	\end{align}
	 Invoking \eqref{eq:max gamma} and  \eqref{eq:final0_multi}, we can write
	 \begin{align}
	 &\nonumber
	 \|\hat{\mathbf{x}}_{i^k}(\tilde{\mathbf{x}}^k)-\mathbf{x}^k_{i^k}\|_2^2
	 \\&\leq\frac{1}{\gamma\left(c_{\tilde{f}}-\gamma\left(L_f+\frac{\delta^2L_f}{2}\right)\right)}\left(\tilde{F}(\mathbf{x}^k,\ldots,\mathbf{x}^{k-\delta})-\tilde{F}(\mathbf{x}^{k+1}\ldots,\mathbf{x}^{k+1-\delta})\right).
	 \end{align}
	 Using this bound in \eqref{final_iter}, we get \vspace{-0.1cm}
	 \begin{align}
	 \nonumber\mathbb{E}\left(\|M_F(\underline{\mathbf{x}}^k)\|_2^2\right)
	 \leq C_1\sum\limits_{t=k}^{k+T}\mathbb{E}\left(\tilde{F}(\underline{\mathbf{x}}^t,\ldots,\underline{\mathbf{x}}^{t-\delta})-\tilde{F}(\underline{\mathbf{x}}^{t+1},\ldots,\underline{\mathbf{x}}^{t+1-\delta})\right)\\
	 \quad+\gamma^2C_2\,\mathbb{E}\left(\sum\limits_{i=1}^N\underline{M}_i^k\sum\limits_{t\in\underline{\mathcal{K}}_i^k}\left(\tilde{F}(\underline{\mathbf{x}}^t,\ldots,\underline{\mathbf{x}}^{t-\delta})-\tilde{F}(\underline{\mathbf{x}}^{t+1},\ldots,\underline{\mathbf{x}}^{t+1-\delta})\right)\right).
	 \end{align}
	 Finally,\vspace{-0.2cm}
	 \begin{align}
	  \nonumber	K_\epsilon\epsilon & \leq\sum\limits_{k=0}^{K_\epsilon}\mathbb{E}\left(\|M_F(\underline{\mathbf{x}}^k)\|_2^2\right)\\
	 &\nonumber	\leq C_1\sum\limits_{k=0}^{K_\epsilon}\mathbb{E}\left(\tilde{F}(\underline{\mathbf{x}}^{k},\ldots,\underline{\mathbf{x}}^{k-\delta})-\tilde{F}(\underline{\mathbf{x}}^{k+T+1},\ldots,\underline{\mathbf{x}}^{k+T+1-\delta})\right)\\
	 &\nonumber	+\gamma^2C_2\,\mathbb{E}\left(\sum\limits_{i=1}^N\underline{M}_i^k\sum\limits_{t\in\underline{\mathcal{K}}_i^k}\left(\tilde{F}(\underline{\mathbf{x}}^t,\ldots,\underline{\mathbf{x}}^{t-\delta})-\tilde{F}(\underline{\mathbf{x}}^{t+1},\ldots,\underline{\mathbf{x}}^{t+1-\delta})\right)\right)\\
	 &\nonumber	\leq C_1(T+1)(F(\mathbf{x}^0)-F^*)\\&\quad +C_2\gamma^2\sum\limits_{k=0}^{K_\epsilon}\mathbb{E}\left(\sum\limits_{i=1}^N\underline{M}_i^k\sum\limits_{t\in\underline{\mathcal{K}}_i^k}\left(\tilde{F}(\underline{\mathbf{x}}^t,\ldots,\underline{\mathbf{x}}^{t-\delta})-\tilde{F}(\underline{\mathbf{x}}^{t+1},\ldots,\underline{\mathbf{x}}^{t+1-\delta})\right)\right).
	 \label{51}
	 \end{align}
This completes the proof.\vspace{-0.4cm}
\subsection{Proof of Theorem \ref{th:ncc}}\label{proof_theorem_2}\vspace{-0.2cm}
In this section, the best-response map $\hat{\mathbf{x}}(\cdot)$  is the one defined in \eqref{ncc_2}.

Statement (ii) of the theorem follow readily from the feasibility of  $\mathbf{x}^0\in\mathcal{K}$ and the fact that $\mathbf{x}^{k+1}_{i^k}=\mathbf{x}^k_{i^k}+\gamma(\hat{\mathbf{x}}_{i^k}(\tilde{\mathbf{x}}^k)-\mathbf{x}^k_{i^k})$ is a convex combinations of points in $\mathcal{K}_{i^k}(\mathbf{x}^k_{i^k})$.  

To prove statement (ii), let us fix a  realization $\omega\in\Omega$. Following the steps from \eqref{descent} to \eqref{eq:final0_multi2}, one can  prove that the following holds a.s.:
\begin{align}
&\nonumber	\mathbb{E}\left(\tilde{F}(\underline{\mathbf{x}}^{k+T}\ldots,\underline{\mathbf{x}}^{k+T-\delta})|\mathcal{F}^{k-1}\right)\leq\tilde{F}(\underline{\mathbf{x}}^k,\ldots,\underline{\mathbf{x}}^{k-\delta})\\&-\gamma\left(c_{\tilde{f}}-\gamma\left(L_f+\frac{\delta^2L_f}{2}\right)\right)\sum\limits_{t=k}^{k+T-1}\mathbb{E}\left(\|\hat{\mathbf{x}}_{\underline{i}^t}(\underline{\tilde{\mathbf{x}}}^t)-\mathbf{x}^t_{\underline{i}^t}\|_2^2|\mathcal{F}^{t-1}\right).
\label{eq:final0_multi2_ncc}
\end{align}

Using \eqref{eq:max gamma},   \eqref{eq:final0_multi2_ncc}    and A5', we deduce  that i) 
$\left\{{F}(\underline{\mathbf{x}}^k,\ldots,\underline{\mathbf{x}}^{k-\delta})\right\}_{k\in\mathbb{N}_+}$ converges a.s., and ii)
\begin{equation}
\lim\limits_{k\to+\infty}\sum\limits_{t=k}^{k+T-1}\mathbb{E}\left(\|\hat{\mathbf{x}}_{\underline{i}^t}(\underline{\tilde{\mathbf{x}}}^t)-\underline{\mathbf{x}}_{\underline{i}^t}^t\|_2|\mathcal{F}^{t-1}\right)=0\quad\text{a.s.}
\label{limit_ncc}
\end{equation}
It follows from  \eqref{limit_ncc} and C2 that
\begin{align}
\lim\limits_{k\to+\infty}\sum\limits_{i=1}^N\|\hat{\mathbf{x}}_i(\tilde{\mathbf{x}}^{k+t_k(i)})-\mathbf{x}_i^{k+t_k(i)}\|_2=0\quad\text{ a.s.}
\label{limit2_ncc}
\end{align}
Therefore,   there exists a set $\bar{\Omega}\subseteq\Omega$, with $\mathbb{P}(\bar{\Omega})=1$, such that, for any $\omega\in\bar{\Omega}$,
\begin{align}
\nonumber &\|\hat{\mathbf{x}}(\mathbf{x}^k)-\mathbf{x}^k\|_2\,\leq\,\sum\limits_{i=1}^N\|\hat{\mathbf{x}}_i(\mathbf{x}^k)
-\mathbf{x}_i^k\|_2\,
\stackrel{(a)}{\leq}\,\sum\limits_{i=1}^N\biggl(\|\hat{\mathbf{x}}_i(\tilde{\mathbf{x}}^{k+t_k(i)})-\mathbf{x}_i^{k+t_k(i)}\|_2\\&\nonumber\quad+\|\tilde{\mathbf{x}}^{k+t_k(i)}-\mathbf{x}^k\|_2^{1/2}\,\left(\|\tilde{\mathbf{x}}^{k+t_k(i)}-\mathbf{x}^k\|_2^{1/2}+\tilde{L}_{\hat{\mathbf{x}}}\biggr)\right)\\
\nonumber&\stackrel{(b)}{\leq}\sum\limits_{i=1}^N\Biggl(\|\hat{\mathbf{x}}_i(\tilde{\mathbf{x}}^{k+t_k(i)})-\mathbf{x}_i^{k+t_k(i)}\|_2+\biggl(\|\mathbf{x}^{k+t_k(i)}-\mathbf{x}^k\|_2^{1/2}\\&\nonumber\quad+\sum\limits_{l=k+t_k(i)-\delta}^{k+t_k(i)-1}\|\mathbf{x}^{l+1}-\mathbf{x}^l\|_2^{1/2}\biggr)\\&\nonumber\quad\biggl(\|\mathbf{x}^{k+t_k(i)}-\mathbf{x}^k\|_2^{1/2}+\sum\limits_{l=k+t_k(i)-\delta}^{k+t_k(i)-1}\|\mathbf{x}^{l+1}-\mathbf{x}^l\|_2^{1/2}+\tilde{L}_{\hat{\mathbf{x}}}\biggr)\Biggr)\nonumber\\
&\stackrel{(c)}{\leq}\sum\limits_{i=1}^N\Biggl(\|\hat{\mathbf{x}}_i(\tilde{\mathbf{x}}^{k+t_k(i)})-\mathbf{x}_i^{k+t_k(i)}\|_2\nonumber\\ &\quad +2\sqrt{\gamma}\sum\limits_{l=k-\delta}^{k+T-1}\|\hat{\mathbf{x}}_{i^l}(\tilde{\mathbf{x}}^l)-\mathbf{x}_{i^l}^l\|_2^{1/2}\,\left(2\sqrt{\gamma}\sum\limits_{l=k-\delta}^{k+T-1}\|\hat{\mathbf{x}}_{i^l}(\tilde{\mathbf{x}}^l)-\mathbf{x}_{i^l}^l\|_2^{1/2}+\tilde{L}_{\hat{\mathbf{x}}}\right)\Biggr),
\label{final_ncc}
\end{align}
where in (a) we used Proposition \ref{Prop_best_response_ncc}; (b) comes from \eqref{x_tilde}; and in (c) we used  the updating rule of the algorithm.
Using \eqref{limit_ncc}, \eqref{limit2_ncc} and \eqref{final_ncc}, we conclude that
\begin{equation}
\lim\limits_{k\to+\infty}\|\hat{\mathbf{x}}(\mathbf{x}^k)-\mathbf{x}^k\|_2=0.
\label{finallimit_ncc}
\end{equation}
A straightforward generalization of  \cite[Theorem 11]{scutari2014distributed} together with   \eqref{finallimit_ncc}  proves that every limit point of $\{\bx^k\}_{k\in \mathbb N_+}$ is  a stationary solution of Problem~\eqref{ncc_1}. Since \eqref{finallimit_ncc} holds for any given realization $\omega\in\bar{\Omega}$, the above results  hold a.s..

Iteration complexity   can be proved  following the steps \eqref{34}-\eqref{51} and using the convexification of the nonconvex constraint sets where needed; details are omitted.
\end{document}